\documentclass[reqno,a4paper]{amsart}

\usepackage{enumerate}

\usepackage[colorlinks=true]{hyperref}
\hypersetup{urlcolor=blue, linkcolor=blue, citecolor=red}

\usepackage[latin9]{inputenc} 
\usepackage[english]{babel}

\newtheorem{theo}{Theorem}[section]

\newtheorem{prop}[theo]{Proposition}

\theoremstyle{definition}
\newtheorem{rmk}[theo]{Remark}

\newcommand\MScN[1]{\href{http://www.ams.org/mathscinet-getitem?mr=#1}{\nolinkurl{(#1)}}}
\newcommand\DOI[1]{\href{http://dx.doi.org/#1}{(doi: \nolinkurl{#1})}}
\newcommand\LINK[1]{\href{#1}{(link: \nolinkurl{#1})}}

\title[Stokes problem with Navier boundary condition]{Maximal $L^p-L^q$ regularity to the Stokes Problem with Navier boundary conditions}
\author{Hind Al Baba}
\address{Laboratoire de Math\'ematiques et de leurs applications-Pau, UMR, CNRS 5142, Batiment IPRA, Universit\'e de Pau et des pays de L'Adour, Avenue de L'universit\'e, Bureau 012, BP 1155, 64013 Pau cedex, France}
\email{hind.albaba@univ-pau.fr}
\address{Institute of Mathematics of the Czech Academy of Sciences, $\breve{Z}$itn\'a 25, 11567 Praha 1, Czech Republic.}
\email{albaba@math.cas.cz}
\thanks{The work of H. Al Baba is done between the Laboratoire de Math\'ematiques et de Leurs Applications; Universit\'e de Pau et des Pays de l'Adour in France and the Departemento de Matem\`aticas, Universidad del Pa\'is Vasco in Bilbao, Spain.}

\begin{document}

\null
\vskip -0.5 truecm

\maketitle

\vspace{-0.5cm}
\begin{center}
\today
\end{center}

\begin{abstract}
We prove in this paper some results on the complex and fractional powers of the Stokes operator with slip frictionless boundary conditions involving the stress tensor. This is fundamental and plays an important role in the associated parabolic problem and will be used to prove maximal $L^{p}-L^{q}$ regularity results for the  non-homogeneous Stokes problem. 
\end{abstract}

\vspace{0.25cm}


\noindent MSC[2008] 35B65, 35D30, 35D35, 35K20, 35Q30,  76D05,  76D07, 76N10.

\section{Introduction}\label{Introduction}
This paper studies maximal $L^p-L^q$ regularity for the Stokes problem with Navier-slip boundary conditions
\begin{eqnarray}
\frac{\partial\boldsymbol{u}}{\partial t} - \Delta \boldsymbol{u
}+\nabla\pi=\boldsymbol{f}, & \mathrm{div}\,\boldsymbol{u}=0 \qquad\qquad\textrm{in}\,\,\,
\Omega\times (0,T),\label{StokesProblem} \\
\boldsymbol{u}\cdot\boldsymbol{n}=0, & \left[\mathbb{D}(\boldsymbol{u})\textbf{\textit{n}}\right]_{\boldsymbol{\tau}}=\boldsymbol{0}\qquad\textrm{on}\,\,\,
\Gamma\times (0,T),\label{Navierbc}\\
\boldsymbol{u}(0)= \boldsymbol{u}_{0}& \qquad \qquad\textrm{in} \,\,\,
\Omega, \label{u0}
\end{eqnarray} 
where $\mathbb{D}(\boldsymbol{u})=\frac{1}{2}(\nabla\boldsymbol{u}+\nabla\boldsymbol{u}^{T})$ is the stress tensor and $\Omega$ is a bounded domain of $\mathbb{R}^{3}$ of class $C^{2,1}$. Here $\boldsymbol{u}$ and $\pi$ denote respectively unknowns velocity filed and pressure of a fluid occupying the domain $\Omega$, while $\boldsymbol{u}_{0}$ and $\boldsymbol{f}$ represent respectively the given initial velocity and the external force. 

In the opinion of engineers and physicists, Systems \eqref{StokesProblem}-\eqref{u0} play an important role in many real life situations, such as in aerodynamics, weather forecast, hemodynamics,... Thus arises naturally the need to carry out a mathematical analysis of these systems which represent the underlying fluid dynamic phenomenology. For some mathematical justifications that stress the use of the boundary conditions \eqref{Navierbc}, we can refer to \cite{Amirat, Mikelic1} for the simulations of flows near rough wall, to \cite{Beavers} in the case of flows near perforated walls and to \cite{Berselli, Pares} for the simulation of turbulent flows.  We note that among the earliest works on the mathematical analysis of the Stokes and Navier-Stokes problems with the Navier-slip boundary conditions \eqref{Navierbc} we can cite the work of V. A.  Solonnikov and V. E. {\v{S}}{\v{c}}adilov  \cite{SolonnikovScadilov} who considered the stationary Stokes problem with the boundary conditions \eqref{Navierbc} in bounded or unbounded domains of $\mathbb{R}^{3}$ and proved the existence and regularity of solutions to this problem.  

The author together with C. Amrouche and A. Rejaiba \cite{Rejaiba}, proved the analyticity of the Stokes semi-group with the boundary conditions \eqref{Navierbc} in $L^{p}$-spaces. This guarantee the existence of complex and fractional powers of the Stokes operator with the corresponding boundary conditions. They also studied the homogeneous Stokes Problem with Navier-slip boundary conditions (\textit{i.e.} Problem \eqref{StokesProblem}-\eqref{u0} with $\boldsymbol{f}=\boldsymbol{0}$) and proved the existence of strong, weak and very weak solutions to this problem. 
 In this paper we shall prove maximal $L^{p}-L^{q}$ regularity for the non-homogeneous case (\textit{i.e.} Problem \eqref{StokesProblem}-\eqref{u0} with $\boldsymbol{f}\neq\boldsymbol{0}$). We shall also prove the existence of strong, weak and very weak solutions to this problem with maximal regularity. The key tool is the use of the complex and fractional powers of the Stokes operator with Navier-slip boundary conditions \eqref{Navierbc}. We note that the concept of very weak solution $\boldsymbol{u}\in\boldsymbol{L}^{p}(\Omega)$ to certain elliptic and parabolic problem with initial data of low regularity was introduced by J. L. Lions and E. Magenes \cite{LM} and is usually based on duality arguments for strong solutions. Therefore the boundary regularity required in this theory is the same as for strong solutions.

\medskip

Concerning the maximal $L^p$ regularity for the Stokes problem we can cite \cite{Solonnikov} by V. A. Solonnikov among the firsts works on this problem. The author constructed in \cite{Solonnikov} a solution $(\boldsymbol{u},\pi)$ to the initial value Stokes problem with Dirichlet boundary conditions ($\boldsymbol{u}=\boldsymbol{0}$ on $\Gamma\times\left[ 0,\,T\right)$). 
His proof was based on methods in the theory of potentials. However, when $\Omega$ is not bounded the result in \cite{Solonnikov} was not global in time. Later on, Y. Giga and H. Sohr  \cite{GiGa4} strengthened Solonnikov's result in two directions. First their result was global in time. Second, the integral norms that they used may have different exponent $p,q$ in space and time. To derive such global maximal $L^{p}-L^{q}$ regularity for the Stokes system with Dirichlet boundary conditions \cite{GiGa4} used the boundedness of the pure imaginary power of the Stokes operator. Exactly they used and extended an abstract perturbation result developed by G. Dore and A. Venni \cite{DV}.

M. Geissert et al. considered \cite{Geissert} the $L^p$ realization of the Hodge Laplacian operator defined by :
$\Delta_{M}:\mathbf{D}(\Delta_{M})\subset\boldsymbol{L}^{p}(\Omega)\longmapsto\boldsymbol{L}^{p}(\Omega),\,$ where 
\begin{equation}\label{DomainHodgeLaplacian}
\mathbf{D}(\Delta_{M})=\big\{
\boldsymbol{u}\in\boldsymbol{W}^{2,p}(\Omega);\,
\boldsymbol{u}\cdot\boldsymbol{n}\,=\,0,\,\,\,\boldsymbol{\mathrm{curl}}\,\boldsymbol{u}\times\boldsymbol{n}\,=\,\boldsymbol{0}\,\,\textrm{on}\,\,\Gamma\big\},
\end{equation}
\begin{equation}\label{HodgeLaplacian}
\forall\,\boldsymbol{u}\in\mathbf{D}(\Delta_{M}),\quad\Delta_{M}\boldsymbol{u}=\Delta\boldsymbol{u}\,\,\,\mathrm{in}\,\,\Omega,
\end{equation}
in a domain $\Omega\subset\mathbb{R}^{3}$ with a suitably smooth boundary. M. Geissert et al. proved \cite{Geissert} that for all $\lambda>0$, the $L^p$ realization of the operator $\lambda I-\Delta_{M}$ admits a bounded $\mathcal{H}^{\infty}$-calculus. They also showed that in the case where $\Omega$ is simply connected their result is true for $\lambda=0.$ Since the class of operators having a bounded $\mathcal{H}^{\infty}$-calculus in their corresponding Banach spaces enjoy the property of bounded pure imaginary powers, they deduced the maximal $L^p-L^q$ regularity to magneto-hydrodynamic equation with perfectly conducting wall condition.

Following the previous results in \cite{Geissert}, the author together with C. Amrouche and M. Escobedo proved in \cite{ARMA}, maximal $L^p-L^q$ regularity to the solution of the inhomogeneous Stokes Problem with  slip frictionless boundary conditions involving the tangential component of the velocity vortex instead of the stress tensor in a domain $\Omega$ not necessarily simply connected. Exactly, the authors considered in \cite{ARMA} the Stokes Problem \eqref{StokesProblem} with the following boundary condition:
\begin{equation}\label{nbc}
\boldsymbol{u}\cdot\boldsymbol{n}=0,\qquad
\boldsymbol{\mathrm{curl}}\,\boldsymbol{u}\times \boldsymbol{n} = \boldsymbol{0}\qquad\textrm{on} \quad \Gamma\times (0,T).
\end{equation}
We note that (cf. \cite{Daveiga}), in the case of a flat boundary, the two conditions \eqref{Navierbc} and \eqref{nbc} are equivalent. 

J. Saal considered in \cite{Jsaal} the Stokes problem in spatial regions with moving boundary and proved maximal $L^{p}-L^{q}$ regularity to this problem. The proof relies on a reduction of the problem to an equivalent nonautonomous system on a cylindrical space-time domain and the result includes bounded and unbounded regions. J. Saal also proved in \cite{Saal} maximal $L^{p}-L^{q}$ regularity for the Stokes problem with homogeneous Robin boundary conditions in the half space $\mathbb{R}^{3}_{+}$. To this end he proved that the associated Stokes operator is sectorial and admits a bounded $\mathcal{H}^{\infty}$-calculus on $\boldsymbol{L}^{p}_{\sigma,\tau}(\Omega)$.  Nevertheless, as shown in \cite{Shimada} the same  approach can not be applied to the Stokes problem with non-homogeneous  Robin boundary condition. For this reason R. Shimada \cite{Shimada} derive the maximal $L^{q}-L^{p}$ regularity for the Stokes problem with non-homogeneous Robin boundary conditions by applying Weis's operator-valued Fourier multiplier theorem to the concrete representation formulas of solutions to the Stokes problem as well as a localization procedure. 

\vspace{0.25cm}

The organization of the paper is as follows. In Section \ref{Stokes-Fractional}  we recall some properties of the Stokes operator with the boundary condition \eqref{Navierbc} that are crucial in our work. We study also the complex and fractional powers of the Stokes operator with the corresponding boundary conditions. We prove our results in the case of a strictly star-shaped domain. In the general case where $\Omega$ is not star-shaped, we recover $\Omega$ (cf. \cite{Ber}),  by a finite number of star open domains of class $C^{2,1}$ and apply the above argument to each of these sets to derive the desired result on the entire domain. 
The proof of this general case is done in Section \ref{appendix}. In Section \ref{Applications} we apply the results of the previous section to Problem \eqref{StokesProblem}-\eqref{u0} and derive maximal  $L^p-L^q$ regularity result to the inhomogeneous Stokes Problem \eqref{StokesProblem}-\eqref{u0}. Finally in Section \ref{appendix}, we show that the result of Section \ref{Stokes-Fractional} hold in a bounded domain $\Omega$ of class $C^{2,1}$ not necessarily star-shaped.

\section{Stokes operator and its fractional powers}\label{Stokes-Fractional}
In this section we give some properties of the Stokes operator with Navier-slip boundary conditions \eqref{Navierbc} as well as its complex and fractional powers. These properties are crucial in the associated parabolic problem and will be used in the sequel to prove maximal regularity results for the inhomogeneous 
Stokes problem. Throughout this paper if we do not state otherwise, $p$ will be a real number such that $1<p<\infty$.

\subsection{Stokes operator}
In this subsection we introduce the different Stokes operators with different regularities in order to obtain strong, weak and very weak solutions to the Stokes Problem \eqref{StokesProblem}.

First we consider the Stokes operator with the boundary conditions \eqref{Navierbc} on  the space $\boldsymbol{L}^{p}_{\sigma,\tau}(\Omega)$ given by 
\begin{equation}\label{Lpsigmatau}
\boldsymbol{L}^{p}_{\sigma,\tau}(\Omega)=\Big\{\boldsymbol{f}\in\boldsymbol{L}^{p}(\Omega);\,\,\,\mathrm{div}\boldsymbol{f}=0\,\,\mathrm{in}\,\,\Omega,\,\, \boldsymbol{f}\cdot\boldsymbol{n}=0\,\,\mathrm{on}\,\,\Gamma\Big\}
\end{equation} 
and we denote it by $A_{p}$. The trace value in \eqref{Lpsigmatau} is justified (see below). Thanks to \cite[Section 3]{Rejaiba} we know that the operator $A_{p}$  is a closed linear densely defined operator on $\boldsymbol{L}^{p}_{\sigma,\tau}(\Omega)$ defined as follows
\begin{equation}\label{c2dpa}
\mathbf{D}(A_{p})= \Big\{ \boldsymbol{u} \in
\boldsymbol{W}^{2,p}(\Omega);\,\,\mathrm{div}\,\boldsymbol{u}=
0\,\, \mathrm{in}\,\,\Omega,
\,\,\boldsymbol{u}\cdot\boldsymbol{n}=0,\,\,\left[\mathbb{D}(\boldsymbol{u})\textbf{\textit{n}}\right]_{\boldsymbol{\tau}}=\boldsymbol{0}
\,\,\mathrm{on}\,\,\Gamma \Big\},
\end{equation}
\begin{equation}\label{defistokesoperator}
\forall\,\boldsymbol{u}\in\mathbf{D}(A_{p}),\quad A_{p}\boldsymbol{u}=-P\Delta\boldsymbol{u}\quad\mathrm{in}\,\,\Omega.
\end{equation}
The operator $P$ in \eqref{defistokesoperator}, is the Helmholtz projection  $P:\boldsymbol{L}^{p}(\Omega)\longmapsto\boldsymbol{L}^{p}_{\sigma,\tau}(\Omega),$ defined by
\begin{equation}\label{helmholtzproj}
\forall\,\boldsymbol{f}\in\boldsymbol{L}^{p}(\Omega),\quad P\boldsymbol{f}\,=\,\boldsymbol{f}\,-\,\boldsymbol{\mathrm{grad}}\,\pi,
\end{equation}
where $\pi\in W^{1,p}(\Omega)/\mathbb{R}$ is the unique solution of the following weak Neuman Problem (cf. \cite{Si}):
\begin{equation}\label{wn.1}
\mathrm{div}\,(\boldsymbol{\mathrm{grad}}\,\pi\,-\,\boldsymbol{f})=0\qquad
\mathrm{in}\,\,\Omega,\qquad
(\boldsymbol{\mathrm{grad}}\,\pi\,-\,\boldsymbol{f})\cdot\boldsymbol{n}=0\qquad\mathrm{on}\,\,\Gamma.
\end{equation} 

An easy computation shows that 
$$\forall\,\,\boldsymbol{u}\in\mathbf{D}(A_{p}),\qquad
A_{p}\boldsymbol{u}\,=\,-\Delta\boldsymbol{u}\,+\boldsymbol{\mathrm{grad}}\,\pi\quad\mathrm{in}\,\,\Omega,$$
where $\pi\in W^{-1,p}(\Omega)/\mathbb{R}$ is the unique solution up to an additive constant of
the problem
\begin{equation*}
\mathrm{div}(\boldsymbol{\mathrm{grad}}\,\pi\,-\,\Delta\boldsymbol{u})=0\qquad\mathrm{in}\,\,\Omega,\qquad
(\boldsymbol{\mathrm{grad}}\,\pi\,-\,\Delta\boldsymbol{u})\cdot\boldsymbol{n}=0\qquad\textrm{on}\,\,\,\Gamma.
\end{equation*}

We note that, (cf. \cite{REJAIBA} and \cite{wata}), in the general case the Stokes operator $A_{p}$ has a non-trivial kernel, we denote this Kernel by $\boldsymbol{\mathcal{T}}^{p}(\Omega)$. When the domain $\Omega$ is obtained by rotation around a vector $\textbf{\textit{b}}\in \mathbb{R}^{3}$, then $$\boldsymbol{\mathcal{T}}^{p}(\Omega)=\mathrm{Span}\{\boldsymbol{\beta}\},\qquad\boldsymbol{\beta}(\textbf{\textit{x}})=\textbf{\textit{b}}\times \textbf{\textit{x}}, \quad \mathrm{for}\, \textbf{\textit{x}}\in\, \Omega.$$
 Otherwise
\begin{equation*}
\boldsymbol{\mathcal{T}}^{p}(\Omega)= \{\boldsymbol{0}\},
\end{equation*}
(see \cite{wata} for more details). The kernel $\boldsymbol{\mathcal{T}}^{p}(\Omega)$ can be characterized as follows (see \cite{REJAIBA}):
\begin{equation}\label{StokesKernelnavier}
\boldsymbol{\mathcal{T}}^{p}(\Omega)=\{\textbf{\textit{v}}\,\in \, \textbf{\textit{W}}\,^{1,p}(\Omega);\,\,\, \mathbb{D}(\textbf{\textit{v}})=\boldsymbol{0}\,\, \mathrm{in} \, \Omega,\,\,\,\mathrm{div}\, \textbf{\textit{v}}=0\,\, \mathrm{in}\,\,\,\Omega \,\,\,\mathrm{and}\,\,\, \textbf{\textit{v}}\cdot \textbf{\textit{n}}=0 \,\,\,\mathrm{on}\,\,\, \Gamma\}.
\end{equation}

 We also note that, (see \cite[Theorem 3.9]{Rejaiba}), the operator $-A_{p}$ is sectorial and generates a bounded analytic semi-group on $\boldsymbol{L}^{p}_{\sigma,\tau}(\Omega),\,$ for all $1<p<\infty$. We denote by $e^{-t A_{p}}$ the analytic semi-group associated to the operator $A_{p}$ in $\boldsymbol{L}^{p}_{\sigma,\tau}(\Omega)$.

\medskip

Next we consider the space 
\begin{equation*}
\boldsymbol{H}^{p}(\textrm{div},\Omega)\,=\,\big\{\boldsymbol{v}\in\boldsymbol{L}^{p}(\Omega);\,\,\,\textrm{div}\,\boldsymbol{v}\in\boldsymbol{L}^{p}(\Omega)\big\},
\end{equation*}
equipped with the graph norm. For every $1<p<\infty$, the space $\boldsymbol{\mathcal{D}}(\overline{\Omega})$ is dense in $\boldsymbol{H}^{p}(\textrm{div},\Omega)$ (cf. \cite[Section 2]{Cherif-Nour-M3AS} and \cite[Proposition 2.3]{Am2}). In addition, for any function $\boldsymbol{v}$ in $\boldsymbol{H}^{p}(\textrm{div},\Omega)$ the normal trace $\boldsymbol{v}\cdot\boldsymbol{n}_{\vert\Gamma}$ exists and belongs to $W^{-1/p,\,p}(\Gamma)$ and that the closure of $\boldsymbol{\mathcal{D}}(\Omega)$ in $\boldsymbol{H}^{p}(\textrm{div},\Omega)$ is equal to
 \begin{equation*}
\boldsymbol{H}^{p}_{0}(\textrm{div},\Omega)\,=\,\big\{\boldsymbol{v}\in\boldsymbol{H}^{p}(\textrm{div},\Omega);\,\,\,\boldsymbol{v}\cdot\boldsymbol{n}\,=\,0\,\,\textrm{on}\,\,\Gamma\big\}.
\end{equation*}
We have denoted by $\boldsymbol{\mathcal{D}}(\Omega)$ the set of infinitely differentiable functions with compact support in $\Omega$ and by $\boldsymbol{\mathcal{D}}(\overline{\Omega})$ the restriction to $\Omega$ of infinitely differentiable functions with compact support in $\mathbb{R}^{3}$. The dual space $[\boldsymbol{H}^{p}_{0}(\textrm{div},\Omega)]'$ of $\boldsymbol{H}^{p}_{0}(\textrm{div},\Omega)$ can be characterized as follows (cf. \cite[Proposition 1.0.4]{Seloulathese}): A distribution $\boldsymbol{f}$ belongs to
$[\boldsymbol{H}^{p}_{0}(\textrm{div},\Omega)]'$
 if and only if there exist
 $\boldsymbol{\psi}\in\boldsymbol{L}^{p'}(\Omega)$ and $\chi\in
 L^{p'}(\Omega)$ such that
 $\boldsymbol{f}\,=\,\boldsymbol{\psi}\,+\,\boldsymbol{\mathrm{grad}}\,\chi$
 and
 \begin{equation*}
\|\boldsymbol{f}\|_{[\boldsymbol{H}^{p}_{0}(\textrm{div},\Omega)]'}\,=\,\inf_{\boldsymbol{f}\,=\,\boldsymbol{\psi}\,+\,\boldsymbol{\mathrm{grad}}\,\chi}\max\,(\|\boldsymbol{\psi}\|_{\boldsymbol{L}^{p'}(\Omega)}\,,\,\|\chi\|_{\boldsymbol{L}^{p'}(\Omega)}).
 \end{equation*}

We consider the following space :
\begin{equation*}
[\boldsymbol{H}^{p'}_{0}(\mathrm{div},\Omega)]'_{\sigma, \tau}=\{\boldsymbol{f}\in[\boldsymbol{H}^{p'}_{0}(\mathrm{div},\Omega)]';\,\,\mathrm{div}\,\boldsymbol{f}=0\quad \mathrm{in}\,\, \Omega \quad \mathrm{and}\quad \boldsymbol{f}\cdot \boldsymbol{n}=0\quad \mathrm{on}\,\,\Gamma\}.
\end{equation*}
We note that (cf. \cite{ARMA}), for a function $\boldsymbol{f}$ in the dual space $[\boldsymbol{H}^{p'}_{0}(\mathrm{div},\Omega)]'$ such that $\,\mathrm{div}\,\boldsymbol{f}\in L^{p}(\Omega),$ the normal trace value $\boldsymbol{f}\cdot\boldsymbol{n}_{\vert\Gamma}$ exists and belongs to the space $\boldsymbol{W}^{-1-1/p,p}(\Gamma).$

The Stokes operator $A_{p}$ can be extended to the space $[\boldsymbol{H}^{p'}_{0}(\mathrm{div},\Omega)]'_{\sigma, \tau}$ (cf. \cite[section 3.2]{Rejaiba}). This extension is a closed linear densely defined operator 
$$B_{p}\,:\,\mathbf{D}(B_{p})\subset[\boldsymbol{H}^{p'}_{0}(\mathrm{div},\Omega)]'_{\sigma, \tau}\longmapsto[\boldsymbol{H}^{p'}_{0}(\mathrm{div},\Omega)]'_{\sigma, \tau},$$ 
\begin{equation}\label{DpB} 
\mathbf{D}(B_{p})=\big\{
\boldsymbol{u}\in\boldsymbol{W}^{1,p}(\Omega);\,\,\,\mathrm{div}\,\boldsymbol{u}=0\,\,\textrm{in}\,\,\Omega,\,\,\,
\boldsymbol{u}\cdot\boldsymbol{n}\,=\,0,\,\,\,\left[\mathbb{D}(\textbf{\textit{u}})\textbf{\textit{n}}\right]_{\boldsymbol{\tau}}\,=\,\boldsymbol{0}\,\,\textrm{on}\,\,\Gamma\big\}
\end{equation}
and 
\begin{equation}\label{operatorB}
\forall\,\,\boldsymbol{u}\in\mathbf{D}(B_{p}),\quad
B_{p}\boldsymbol{u}\,=\,-\Delta\boldsymbol{u}\,+\boldsymbol{\mathrm{grad}}\,\pi \quad\mathrm{in}\,\,\Omega,
\end{equation}
where $\,\pi\in L^{p}(\Omega)/\mathbb{R}$ is the unique solution up to an additive constant of
the problem
\begin{equation*}
\mathrm{div}(\boldsymbol{\mathrm{grad}}\,\pi\,-\,\Delta\boldsymbol{u})=0\qquad\mathrm{in}\,\,\Omega,\qquad
(\boldsymbol{\mathrm{grad}}\,\pi\,-\,\Delta\boldsymbol{u})\cdot\boldsymbol{n}=0\qquad\textrm{on}\,\,\,\Gamma.
\end{equation*}
The operator $-B_{p}$ generates a bounded analytic semi-group on $\,[\boldsymbol{H}^{p'}_{0}(\mathrm{div},\Omega)]'_{\sigma, \tau},\,$ for all $\,1<p<\infty$ (see \cite[Theorem 3.10]{Rejaiba}).
We note that the trace value $\left[\mathbb{D}(\textbf{\textit{u}})\textbf{\textit{n}}\right]_{\boldsymbol{\tau}}$ for a function $\boldsymbol{u}$ in \eqref{DpB} is justified by that, ( see \cite[Lemma 2.4]{REJAIBA}), for a function $\boldsymbol{u}\in\boldsymbol{W}^{1,p}(\Omega)$ such that $\Delta\boldsymbol{u}\in[\boldsymbol{H}^{p'}_{0}(\mathrm{div},\Omega)]',$ the trace value $\left[\mathbb{D}(\textbf{\textit{u}})\textbf{\textit{n}}\right]_{\boldsymbol{\tau}}$ exists and belongs to $\boldsymbol{W}^{\frac{-1}{p},p}(\Gamma)$. 

\medskip

Consider also the following space: 
\begin{equation}\label{tp}
\boldsymbol{T}^{p}(\Omega)=\big\{\boldsymbol{v}\in\boldsymbol{H}^{p}_{0}(\mathrm{div},\Omega);\,\,\,\mathrm{div}\,\boldsymbol{v}\in
{W}^{1,p}_{0}(\Omega)\big\}.
\end{equation}
Thanks to \cite[Lemmas 4.11, 4.12]{Amrouche1} we
know that $\boldsymbol{\mathcal{D}}(\Omega)$ is dense in
$\boldsymbol{T}^{p}(\Omega)$ and a distribution $\boldsymbol{f}\in(\boldsymbol{T}^{p}(\Omega))'$ if and only if there exist $\boldsymbol{\psi}\in\boldsymbol{L}^{p'}(\Omega)$ and $\mathit{f}_{0}\in W^{-1,p'}(\Omega)$ such that $\boldsymbol{f}=\boldsymbol{\psi}\,+\,\nabla\mathit{f}_{0},\,$ with
 \begin{equation*}
\Vert\boldsymbol{f}\Vert_{(\boldsymbol{T}^{p}(\Omega))'}=\min_{\boldsymbol{f}=\boldsymbol{\psi}\,+\,\nabla\mathit{f}_{0}}\max(\,\Vert\boldsymbol{\psi}\Vert_{\boldsymbol{L}^{p'}(\Omega)}\,,\,\Vert\mathit{f}_{0}\Vert_{W^{-1,p'}(\Omega)}). 
 \end{equation*}
We consider the following subspace
\begin{equation*}
[\boldsymbol{T}^{p'}(\Omega)]'_{\sigma, \tau}=\{\boldsymbol{f}\in(\boldsymbol{T}^{p'}(\Omega))';\,\,\mathrm{div}\,\boldsymbol{f}=0\quad \mathrm{in}\,\, \Omega \quad \mathrm{and}\quad \boldsymbol{f}\cdot \boldsymbol{n}=0\quad \mathrm{on}\,\,\Gamma\}.
\end{equation*}
We recall that (cf. \cite{ARMA}), for a function $\boldsymbol{f}$ in the dual space $[\boldsymbol{T}^{p'}(\Omega)]'$ such that $\mathrm{div}\,\boldsymbol{f}\in L^{p}(\Omega)$ the normal trace $\boldsymbol{f}\cdot\boldsymbol{n}_{\vert\Gamma}$ exists and belongs to $\,\boldsymbol{W}^{-2-1/p,p}(\Gamma)\,.$

The Stokes operator with Navier slip boundary condition can also be extended to the space $[\boldsymbol{T}^{p'}(\Omega)]'_{\sigma, \tau}$ (see \cite[Section 3.3]{Rejaiba}). This extension is a densely defined closed linear operator
$$\,C_{p}\,:\,\mathbf{D}(C_{p})\subset [\boldsymbol{T}^{p'}(\Omega)]'_{\sigma, \tau}\longmapsto[\boldsymbol{T}^{p'}(\Omega)]'_{\sigma, \tau},\quad\mathrm{where}$$ 
\begin{multline}\label{DpC} 
\mathbf{D}(C_{p})=\big\{
\boldsymbol{u}\in\boldsymbol{L}^{p}(\Omega);\,\,\,\mathrm{div}\,\boldsymbol{u}=0\,\,\textrm{in}\,\,\Omega,\,\,
\boldsymbol{u}\cdot\boldsymbol{n}\,=\,0,\,\,\,\left[\mathbb{D}(\textbf{\textit{u}})\textbf{\textit{n}}\right]_{\boldsymbol{\tau}}\,=\,\boldsymbol{0}\,\,\textrm{on}\,\,\Gamma\big\}
\end{multline}
 and for all $\,\boldsymbol{u}\in\mathbf{D}(C_{p}),\,\,
C_{p}\boldsymbol{u}\,=\,-\Delta\boldsymbol{u}\,+\boldsymbol{\mathrm{grad}}\,\pi\,$ in $\Omega,$
with $\pi\in W^{-1,p}(\Omega)/\mathbb{R}$ the unique solution up to an additive constant of 
the problem
\begin{equation*}\label{stokesnavierveryweak}
\mathrm{div}(\boldsymbol{\mathrm{grad}}\,\pi\,-\,\Delta\boldsymbol{u})=0\qquad\mathrm{in}\,\,\Omega,\qquad
(\boldsymbol{\mathrm{grad}}\,\pi\,-\,\Delta\boldsymbol{u})\cdot\boldsymbol{n}=0\qquad\textrm{on}\,\,\,\Gamma.
\end{equation*}
The operator $-C_{p}$ generates a bounded analytic semi-group on $[\boldsymbol{T}^{p'}(\Omega)]'_{\sigma, \tau}$ for all $1<p<\infty$ (see \cite[Theorem 3.12]{Rejaiba}). We recall that, (see \cite[Lemma 5.4]{REJAIBA}), for a function $\boldsymbol{u}\in\boldsymbol{L}^{p}(\Omega)$ such that $\Delta\boldsymbol{u}\in (\boldsymbol{T}^{p'}(\Omega))',$ the trace value $\left[\mathbb{D}(\textbf{\textit{u}})\textbf{\textit{n}}\right]_{\boldsymbol{\tau}}$ exists and belongs to $\boldsymbol{W}^{-1-\frac{1}{p},p}(\Gamma)$. This give a meaning to the trace $\left[\mathbb{D}(\textbf{\textit{u}})\textbf{\textit{n}}\right]_{\boldsymbol{\tau}}$ of a function $\boldsymbol{u}$ in \eqref{DpC}.

\vspace{0.5cm}

\subsection{Fractional powers of the Stokes operator}
This subsection is devoted to the study of the complex and the fractional powers of the Stokes operators $A_{p}$ on $\boldsymbol{L}^{p}_{\sigma,\tau}(\Omega)$.  Since the Stokes operator $A_{p}$ in $\boldsymbol{L}^{p}_{\sigma,\tau}(\Omega)$ generates a bounded analytic semi-group, it is in particular a non-negative operator. It then follows from the results in \cite{Ko} and in \cite{Tri} that its complex and fractional powers $A^{\alpha}_{p}$, $\alpha\in\mathbb{C}$, are well, densely defined and  closed linear operators on $\boldsymbol{L}^{p}_{\sigma,\tau}(\Omega)$ with domain $\mathbf{D}(A^{\alpha}_{p})$. Furthermore, using the fact that 
\begin{equation*}
\forall\alpha\in\mathbb{C},\qquad\boldsymbol{\mathcal{D}}_{\sigma}(\Omega)\hookrightarrow\mathbf{D}(A^{\alpha}_{p})\hookrightarrow\boldsymbol{L}^{p}_{\sigma,\tau}(\Omega),
\end{equation*}
one has the density of $\boldsymbol{\mathcal{D}}_{\sigma}(\Omega)$ in $\mathbf{D}(A^{\alpha}_{p})$ for all $\alpha\in\mathbb{C}$. We have denoted by $\mathbb{C}$ the set of complex number and $\mathbb{C}^{\ast}=\mathbb{C}\setminus \{0\}$.

Nevertheless as described above since the Stokes operator $A_{p}$ is not invertible with bounded inverse, its complex powers can not be expressed through an integral formula and it is not easy in general to compute calculus inequality involving these powers. To avoid this difficulty we prove the desired results for the operator $(I+A_{p})$. We start by the following proposition


\begin{prop}\label{I+lap}
$\mathbf{(i)}$ There exists an angle $0<\theta_{0}<\pi/2$  and a constant $\kappa(\Omega,p)$ depending only on $\Omega$ and $p$ such that 
the resolvent set of the operator $-(I+A_{p})$ contains the sector
\begin{equation}\label{sector}
\Sigma_{\theta_{0}}=\big\{\lambda\in\mathbb{C};\,\,\,\vert\arg\lambda\vert\leq\pi-\theta_{0}\big\}.
\end{equation}
Moreover the following estimate holds 
\begin{equation}\label{estplus}
\forall\,\lambda\in\Sigma_{\theta_{0}},\,\,\,\lambda\neq 0,\qquad \Vert(\lambda\,I+I+A_{p})^{-1}\Vert_{\mathcal{L}(\boldsymbol{L}^{p}_{\sigma,\tau}(\Omega))}\leq\,\frac{\kappa(\Omega,p)}{\vert\lambda\vert}.
\end{equation}

\noindent $\mathbf{(ii)}$ Let $0<\alpha<1$ be fixed, then for all $\lambda\in\Sigma_{\theta_{0}}$ such that $\lambda\neq 0$ and $\vert\lambda\vert\leq\frac{1}{2\,\kappa(\Omega,p)}$ one has
\begin{equation}\label{estalpha}
\Vert(\lambda\,I+I+A_{p})^{-1}\Vert_{\mathcal{L}(\boldsymbol{L}^{p}_{\sigma,\tau}(\Omega))}\leq\,2^{\alpha}\,\kappa^{\alpha}(\Omega,p)\,\vert\lambda\vert^{\alpha-1}.
\end{equation}
\end{prop}

\begin{proof}
\textbf{(i) } Since the operator $-A_{p}$ generates a bounded analytic semi-group on $\boldsymbol{L}^{p}_{\sigma,\tau}(\Omega)$, then the operator $I+A_{p}$ is an isomorphism from $\mathbf{D}(A_{p})\subset\boldsymbol{L}^{p}_{\sigma,\tau}(\Omega)$ in $\boldsymbol{L}^{p}_{\sigma,\tau}(\Omega)$. We recall that $\mathbf{D}(A_{p})$ is given by \eqref{c2dpa}. Let $\lambda\in\mathbb{C}^{\ast}$ such that $\mathrm{Re}\,\lambda\geq 0$. It is clear that the operator $\lambda\,I+I+A_{p}$ is an isomorphism from $\mathbf{D}(A_{p})$ to $\boldsymbol{L}^{p}_{\sigma,\tau}(\Omega)$. Using \cite[Theorem 3.8]{Rejaiba}, one has since $\mathrm{Re}\,\lambda\geq 0$ 
\begin{equation}\label{estresolI+Ap}
\Vert(\lambda\,I+I+A_{p})^{-1}\Vert_{\mathcal{L}(\boldsymbol{L}^{p}_{\sigma,\tau}(\Omega))}\,\leq\,\frac{\kappa(\Omega,p)}{\vert\lambda\,+\,1\vert}\,\leq\,\frac{\kappa(\Omega,p)}{\vert\lambda\vert},
\end{equation}
where the constant $\kappa(\Omega,p)$ is independent of $\lambda$.
This means that the resolvent set of the operator $-(I+A_{p})$ contains the set $\{\lambda\in\mathbb{C}^{\ast};\,\,\mathrm{Re}\,\lambda\geq 0\}$ where the estimate (\ref{estresolI+Ap}) is satisfied . Using the result of \cite[Chapter VIII, Theorem 1]{Yo}, we deduce that  there exists an angle $0<\theta_{0}<\pi/2$, such that the resolvent set of $-(I+A_{p})$ contains the sector $\Sigma_{\theta_{0}}$. In addition for every $\lambda\in\Sigma_{\theta_{0}}$ such that $\lambda\neq 0$ estimate (\ref{estplus}) holds.

\medskip

\noindent\textbf{(ii)} Let $0<\alpha<1$ be fixed and $\lambda\in\Sigma_{\theta_{0}}$ such that $\lambda\neq 0$ and $\vert\lambda\vert\leq\,\frac{1}{2\,\kappa(\Omega,p)}$. 
Let also $\boldsymbol{f}\in\boldsymbol{L}^{p}_{\sigma,\tau}(\Omega)\,$ and $\boldsymbol{u}\in\mathbf{D}(A_{p})$ solution to $(\lambda\,I+I+A_{p})^{-1}\boldsymbol{f}=\boldsymbol{u}$. Observe that $\boldsymbol{u}$ satisfies
\begin{equation*}
\left\{
\begin{array}{cccc}
\boldsymbol{u}-\Delta\boldsymbol{u}+\nabla\pi=\boldsymbol{f}-\lambda\,\boldsymbol{u},&\mathrm{div}\,\boldsymbol{u}=0&\textrm{in}&\Omega,\\
\boldsymbol{u}\cdot\boldsymbol{n}=0, & \left[\mathbb{D}(\boldsymbol{u})\textbf{\textit{n}}\right]_{\boldsymbol{\tau}}=\boldsymbol{0}&\textrm{on}&\Gamma.
\end{array}
\right.
\end{equation*}
Thus using \cite[Theorem 3.8]{Rejaiba} we have
\begin{equation*}
\Vert\boldsymbol{u}\Vert_{\boldsymbol{L}^{p}(\Omega)}\,\leq\,\kappa(\Omega,p)\,\Vert\boldsymbol{f}-\lambda\,\boldsymbol{u}\Vert_{\boldsymbol{L}^{p}(\Omega)}.
\end{equation*}
We also know from \cite[Remark 3.5]{Rejaiba} that
\begin{equation*}
\|\nabla \pi\|_{\textbf{\textit{L}}^{\,p}(\Omega)}\leq C \|\boldsymbol{u}\|_{\boldsymbol{W}^{\,1,p}(\Omega)}.
\end{equation*}
Due to our assumption on $\lambda$ and the fact that $2\,\kappa(\Omega,p)\vert\lambda\vert\leq\,1$, we have
\begin{equation*}
\Vert\boldsymbol{u}\Vert_{\boldsymbol{L}^{p}(\Omega)}\,\leq\,2\,\kappa(\Omega,p)\,\Vert\boldsymbol{f}\Vert_{\boldsymbol{L}^{p}(\Omega)}\,\leq\,2^{\alpha}\,\kappa^{\alpha}(\Omega,p)\,\vert\lambda\vert^{\alpha-1}\,\Vert\boldsymbol{f}\Vert_{\boldsymbol{L}^{p}(\Omega)},
\end{equation*}
which ends the proof.
\end{proof}

Next we state and prove our results on the pure imaginary powers of the operator $I+A_{p}$

\begin{theo}\label{pureimg1+lap}
Let $\theta_{0}$ as in Proposition \ref{I+lap}. There exist a constant $M>0$ such that for all $s\in\mathbb{R}$ we have
\begin{equation}\label{estimpur1+lap}
\Vert(I+A_{p})^{i\,s}\Vert_{\mathcal{L}(\boldsymbol{L}^{p}_{\sigma,\tau}(\Omega))}\,\leq\,M\,\kappa(\Omega,p)\,e^{\vert s\vert\,\theta_{0}}.
\end{equation}
\end{theo}
\begin{proof} The proof is done in two steps.

\noindent\textbf{(i)} Let $0<\alpha<1$ is fixed and $z\in\mathbb{C}$ such that $-1<-\alpha<\mathrm{Re}\,z<0$.  Thanks to Proposition \ref{I+lap} we know that the operator $I+A_{p}$ is a non-negative bounded operator with bounded inverse. As a result its complex powers can be expressed through the following Dunford integral formula (cf. \cite{Ko}):
\begin{equation}\label{forintimpur}
(I+A_{p})^{z}\,=\,\frac{1}{2\,\pi\,i}\,\int_{\Gamma_{\theta_{0}}}(-\lambda)^{z}\,(\lambda\,I+I+A_{p})^{-1}\,\textrm{d}\,\lambda,
\end{equation}
where
\begin{equation}\label{gammatheta0}
\Gamma_{\theta_{0}}\,=\,\{\rho\,e^{i\,(\pi-\theta_{0})};\,\,0\leq\rho\leq\infty\}\,\cup\,\{-\,\rho\,e^{i\,(\theta_{0}-\pi)};\,\,-\infty\leq\rho\leq 0\}.
\end{equation}
This means that 
\begin{multline*}
(I+A_{p})^{z}\,=\,\frac{1}{2\,\pi\,i}\,\Big[\int_{0}^{+\infty}\big(-\,\rho\,e^{i\,(\pi-\theta_{0})}\big)^{z}\big(\rho\,e^{i\,(\pi-\theta_{0})}\,I\,+\,I+A_{p}\big)^{-1}\,e^{i\,(\pi-\theta_{0})}\,\textrm{d}\rho\\-\,\int_{0}^{+\infty}\big(-\,\rho\,e^{i\,(\theta_{0}-\pi)}\big)^{z}\big(\rho\,e^{i\,(\theta_{0}-\pi)}I\,+I+A_{p}\big)^{-1}\,e^{i\,(\theta_{0}-\pi)}\textrm{d}\rho\,\Big].
\end{multline*}
In addition, we know that
$(-\lambda)^{z}\,=\,e^{z(\ln\vert\lambda\vert\,+\,i\,\mathrm{Arg}(-\lambda) )}$,
where $\mathrm{Arg}(-\lambda)$ is the principal argument of $-\lambda$. An easy computation shows that
\begin{equation}
\vert(-\lambda)\vert^{z}\,\leq\,\rho^{\mathrm{Re}\,z}\,e^{\vert\mathrm{Im}\,z\vert\,\theta_{0}}.
\end{equation}
As a result we have
\begin{equation}\label{complexpower2}
\Vert(I+A_{p})^{z}\Vert_{\mathcal{L}(\boldsymbol{L}^{p}_{\sigma,\tau}(\Omega))}\,\leq\,\frac{e^{\vert\mathrm{Im}\,z\vert\,\theta_{0}}}{2\,\pi}\,\Big[\,I_{1}\,+\,I_{2}\,\Big],
\end{equation}
with
\begin{equation*}
I_{1}\,=\,\int_{0}^{+\infty}\rho^{\mathrm{Re}\,z}\,\Vert\big(\rho\,e^{i\,(\pi-\theta_{0})}\,I\,+\,I+A_{p}\big)^{-1}\Vert_{\mathcal{L}(\boldsymbol{L}^{p}_{\sigma,\tau}(\Omega))}\,\textrm{d}\rho
\end{equation*}
and
\begin{equation*}
I_{2}\,=\,\int_{0}^{+\infty}\rho^{\mathrm{Re}\,z}\,\Vert\big(\rho\,e^{i\,(\theta_{0}-\pi)}\,I\,+\,I+A_{p}\big)^{-1}\Vert_{\mathcal{L}(\boldsymbol{L}^{p}_{\sigma,\tau}(\Omega))}\,\textrm{d}\rho.
\end{equation*}
Next, we write $I_{1}$ in the form
\begin{multline*}
I_{1}\,=\,\int_{0}^{1/2\,\kappa(\Omega,p)}\rho^{\mathrm{Re}\,z}\,\Vert\big(\rho\,e^{i\,(\pi-\theta_{0})}\,I\,+\,I+A_{p}\big)^{-1}\Vert_{\mathcal{L}(\boldsymbol{L}^{p}_{\sigma,\tau}(\Omega))}\,\textrm{d}\rho\\
+\,\int_{1/2\,\kappa(\Omega,p)}^{+\infty}\rho^{\mathrm{Re}\,z}\,\Vert\big(\rho\,e^{i\,(\pi-\theta_{0})}\,I\,+\,I+A_{p}\big)^{-1}\Vert_{\mathcal{L}(\boldsymbol{L}^{p}_{\sigma,\tau}(\Omega))}\,\textrm{d}\rho.
\end{multline*}
As a consequence, thanks to Proposition \ref{I+lap} part (ii) we have
\begin{equation*}
I_{1}\,\leq\,2^{\alpha}
\,\kappa^{\alpha}_{1}(\Omega,p)\,\int_{0}^{1/2\,\kappa(\Omega,p)}\frac{\textrm{d}\rho}{\rho^{1\,-\,\alpha\,-\,\mathrm{Re}\,z}}\,+\,\kappa(\Omega,p)\,\int_{1/2\,\kappa(\Omega,p)}^{+\infty}\frac{\textrm{d}\rho}{\rho^{1\,-\,\mathrm{Re}\,z}}.
\end{equation*}
Thanks to our assumption on $z$ we can verify that 
$$I_{1}\,<\,C(\alpha,\Omega,p), \qquad I_{2}\,<\,C(\alpha,\Omega,p),$$
 with a constant $C(\alpha,\Omega,p)$ depending on $\alpha,\,\Omega,\,p$ and independent of $\mathrm{Re}\,z$. Finally substituting in (\ref{complexpower2}) we have
\begin{equation}\label{complexpower3}
\Vert(I+A_{p})^{z}\Vert_{\mathcal{L}(\boldsymbol{L}^{p}_{\sigma,\tau}(\Omega))}\,\leq\,C\,e^{\vert\mathrm{Im}\,z\vert\,\theta_{0}},
\end{equation}
with a constant $C$ depending on $\alpha,\,\Omega,\,p$ and independent of $\mathrm{Re}\,z$.

\noindent\textbf{(ii)} Let $s\in\mathbb{R},$ to obtain estimate \eqref{estimpur1+lap}, we use the fact that for all $\boldsymbol{f}\in\mathbf{D}(A_{p}),$ the operator $(I+A_{p})^{z}\boldsymbol{f}$ is analytic in $z$, for $-1<\mathrm{Re}z<1$ (see \cite[Propositions 4.7, 4.10]{Ko}). 
Finally using the density of $\mathbf{D}(A_{p})$ in $\boldsymbol{L}^{p}_{\sigma,\tau}(\Omega)$ we obtain estimate \eqref{estimpur1+lap} for all $\boldsymbol{f}\in\boldsymbol{L}^{p}_{\sigma,\tau}(\Omega)$ .
\end{proof}

The following theorem extends Theorem \ref{pureimg1+lap} to the operator $\lambda I+A_{p}$ for all $\lambda>0$.
\begin{theo}\label{Lapimpowerproposition}
There exist an angle $\theta_{0}\in (0,\frac{\pi}{2})$ and a constant $M>0$ such that, for all $s\in  \mathbb{R}$ and for all $\lambda>0$:
\begin{equation}\label{estnuDeltaM*}
\Vert(\lambda\,I+A_{p})^{i\,s}\Vert_{\mathcal{L}(\boldsymbol{L}^{p}_{\sigma,\tau}(\Omega))}\,\leq\,M\,e^{\vert s\vert\,\theta_{0}}, 
\end{equation}
where $M$ is independent of $\lambda$.
\end{theo}
\begin{proof}
Suppose first that $\Omega$ is strictly star shaped with respect to one of its points, then after translation in $\mathbb{R}^{3}$, we can suppose that this point is $0$. This amounts to say that 
\begin{equation*}
\forall\,\mu>1,\qquad\mu\,\overline{\Omega}\subset\Omega,\,\,\,\forall\,0\leq\mu<1\qquad\textrm{and}\,\,\,\,\overline{\Omega}\subset\mu\,\Omega.
\end{equation*} 
Here we take $\mu>1$ and we set $\Omega_{\mu}=\mu\,\Omega$. The proof is based on the use of the following scaling transformation 
\begin{equation}\label{SMU}
\forall\,x\in\Omega_{\mu},\qquad(S_{\mu}\boldsymbol{f})(x)=\boldsymbol{f}(x/\mu),\qquad\boldsymbol{f}\in\boldsymbol{L}^{p}(\Omega).
\end{equation}
Notice that 
\begin{equation*}
\mu^{2}A_{p}\,=\,S_{\mu}A_{p}\,S_{\mu}^{-1},\qquad I+\mu^{2}A_{p}\,=\,S_{\mu}(I+A_{p})\,S^{-1}_{\mu}.
\end{equation*}
Using then the integral representation \eqref{forintimpur} we can verify that,
\begin{equation*}
\forall\,z\in\mathbb{C},\quad(I+\mu^{2}\,A_{p})^{z}\,=\,S_{\mu}
(I+A_{p})^{z}S^{-1}_{\mu}.
\end{equation*}
As a result, for all $z\in\mathbb{C}$, we have
$$\Vert(I+\mu^{2}A_{p})^{z}\Vert_{\mathcal{L}(\boldsymbol{L}^{p}_{\sigma,\tau}(\Omega))}=\Vert S_{\mu}
(I+A_{p})^{z}S^{-1}_{\mu}\Vert_{\mathcal{L}(\boldsymbol{L}^{p}_{\sigma,\tau}(\Omega))}\leq\Vert(I+A_{p})^{z}\Vert_{\mathcal{L}(\boldsymbol{L}^{p}_{\sigma,\tau}(\Omega))}.$$
Using Theorem \ref{pureimg1+lap}, we deduce that there exist $0<\theta_{0} <\pi/2$ and constants $M>0$ such that : 
\begin{equation}\label{complexpower4}
\forall\,s\in\mathbb{R},\quad\Vert(I+\mu^{2}\,A_{p})^{i\,s}\Vert_{\mathcal{L}(\boldsymbol{L}^{p}_{\sigma,\tau}(\Omega))}\,\leq\,M\,e^{\vert s\vert\,\theta_{0}},
\end{equation}
where the constants $M$ in \eqref{complexpower4} is independent of $\mu$.

Observe that
\begin{equation*}
\Big(\frac{1}{\mu^{2}}\,I+A_{p}\Big)^{i\,s}\,=\,\frac{1}{\mu^{2\,i\,s}}\,(I+\mu^{2}A_{p})^{i\,s}.
\end{equation*}
As a result \eqref{estnuDeltaM*} follows.

In the general case, for a domain $\Omega\,$ of Class $C^{2,1},\,$ we use the fact that (see \cite{Ber} for instance), a bounded Lipschitz-Continuous open set is the union of a finite number of star-shaped, Lipschitz-continuous open sets. Clearly, It suffices to apply the above argument to each of these sets to derive the desired result on the entire domain. However, the divergence-free condition of a function $\boldsymbol{f}\in\boldsymbol{L}^{p}_{\sigma,\tau}(\Omega)$ is
not preserved under the cut-off procedure and this process is non-trivial. The proof of this general case is done in Section \ref{appendix} below. 
\end{proof}

The following theorem extend Theorems \ref{pureimg1+lap} and \ref{Lapimpowerproposition} to the operators $(\lambda I+B_{p})$ and $(\lambda I+C_{p}),$ $\lambda>0$, on $[\boldsymbol{H}^{p'}_{0}(\mathrm{div},\Omega)]'_{\sigma,\tau}$ and $[\boldsymbol{T}^{p'}(\Omega)]'_{\sigma,\tau}$ respectively. This result will be used in Section \ref{Applications} in order to obtain weak ans very weak solutions solutions to Problem \eqref{StokesProblem}-\eqref{u0}.  
\begin{theo}\label{PureimgI+Bp+Cp}
There exists $0<\theta_{0}<\pi/2$ and a constant $C>0$ such that for all $\lambda>0$ and for all $s\in\mathbb{R}$
\begin{equation}\label{pureimphdiv}
\Vert(\lambda I+B_{p})^{i\,s}\Vert_{\mathcal{L}([\boldsymbol{H}^{p'}_{0}(\mathrm{div},\Omega)]'_{\sigma,\tau})}\,\leq\,C\,e^{\vert s\vert\,\theta_{0}},
\end{equation}
\begin{equation}\label{pureimptp}
\Vert(\lambda I+C_{p})^{i\,s}\Vert_{\mathcal{L}([\boldsymbol{T}^{p'}(\Omega)]'_{\sigma,\tau})}\,\leq\,C\,e^{\vert s\vert\,\theta_{0}},
\end{equation}
where the constant $C$ in \eqref{pureimphdiv} and \eqref{pureimptp} is independent of $\lambda$.
\end{theo}
\begin{proof}
 It suffices to prove estimate \eqref{pureimphdiv}, estimate \eqref{pureimptp} follows in the same way. Using Theorem \ref{Lapimpowerproposition} one has for all $\lambda>0$ and for all $\boldsymbol{f}\in\boldsymbol{L}^{p}_{\sigma,\tau}(\Omega)$ :
\begin{equation}
\Vert(\lambda I+B_{p})^{i\,s}\boldsymbol{f}\Vert_{[\boldsymbol{H}^{p'}_{0}(\mathrm{div},\Omega)]'}=\Vert(\lambda I+A_{p})^{i\,s}\boldsymbol{f}\Vert_{[\boldsymbol{H}^{p'}_{0}(\mathrm{div},\Omega)]'}\leq  C\,e^{\vert s\vert\,\theta_{0}}\Vert\boldsymbol{f}\Vert_{\boldsymbol{L}^{p}(\Omega)}.
\end{equation}
Next, using the density of $\boldsymbol{L}^{p}_{\sigma,\tau}(\Omega)$ in $[\boldsymbol{H}^{p'}_{0}(\mathrm{div},\Omega)]'_{\sigma,\tau}$ (see \cite[Proposition 3.9]{ARMA}) and the Hahn-Banach theorem we can extend $(\lambda I+B_{p})^{i\,s}$ to a bounded linear operator on $[\boldsymbol{H}^{p'}_{0}(\mathrm{div},\Omega)]'_{\sigma,\tau}$ and we deduce deduce estimate \eqref{pureimphdiv}.
\end{proof}

In the case where the domain $\Omega$ is not obtained by rotation around a vector $\textbf{\textit{b}}\in \mathbb{R}^{3}$, the Stokes operator with Navier slip boundary conditions is invertible with bounded inverse. In this case we can (c.f. \cite[Lemma A2]{GiGa4}) pass to the limit in \eqref{estnuDeltaM*}, \eqref{pureimphdiv} and \eqref{pureimptp} as $\lambda$ tends to zero. As a result we deduce the following theorem.
\begin{theo}\label{PureimgApbpcp}
Suppose that the domain $\Omega$ is not obtained by rotation around a vector $\textbf{\textit{b}}\in \mathbb{R}^{3}$. There exists $0<\theta_{0}<\pi/2$ and a constant $C>0$ such that for all $\lambda>0$ and for all $s\in\mathbb{R}$
\begin{equation*}
\Vert(A_{p})^{i\,s}\Vert_{\mathcal{L}(\boldsymbol{L}^{p}_{\sigma,\tau}(\Omega))}\,\leq\,M\,e^{\vert s\vert\,\theta_{0}}, 
\end{equation*}
\begin{equation*}
\Vert(B_{p})^{i\,s}\Vert_{\mathcal{L}([\boldsymbol{H}^{p'}_{0}(\mathrm{div},\Omega)]'_{\sigma,\tau})}\,\leq\,C\,e^{\vert s\vert\,\theta_{0}},
\end{equation*}
\begin{equation*}
\Vert(C_{p})^{i\,s}\Vert_{\mathcal{L}([\boldsymbol{T}^{p'}(\Omega)]'_{\sigma,\tau})}\,\leq\,C\,e^{\vert s\vert\,\theta_{0}},
\end{equation*}
\end{theo}

\medskip

Next we study the domains of fractional powers of the operator $A_{p}$ on $\boldsymbol{L}^{p}_{\sigma,\tau}(\Omega)$.  Since the Stokes operator with the boundary conditions \eqref{Navierbc} doesn't have bounded inverse, attention should be paid in the calculus of the domains $\mathbf{D}(A^{\alpha}_{p})$ and their norms. 
It follows from \cite{Ko} that for $\mathrm{Re}\alpha>0$, the domain $\mathbf{D}(\nu\,I\,+\, A^{\alpha}_{p})$ doesn't depend on $\nu\geq0$ and coincides with $\mathbf{D}(\mu\,I\,+\,A^{\alpha}_{p})$ for $\mu\geq0$. Exactly 
\begin{equation*}
\forall\,\mu,\,\nu>0,\qquad \mathbf{D}(A^{\alpha}_{p})\,=\,\mathbf{D}(\mu\,I\,+\,A^{\alpha}_{p})\,=\,\mathbf{D}(\nu\,I\,+\,A^{\alpha}_{p}).
\end{equation*}
We also know from \cite[Theorem 1.15.3]{Tri} that the boundedness of the pure imaginary powers of the operator $(I+A_{p})$ allows us to determine the domain of definition of $\mathbf{D}(I+A_{p})^{\alpha})$, and then of $\mathbf{D}(A^{\alpha}_{p})$ for complex number $\alpha$ satisfying $\mathrm{Re}\,\alpha>0$ using complex interpolation theory. In addition for all $\alpha>0$, the map $\boldsymbol{v}\longmapsto\Vert (I+A_{p})^{\alpha}\,\boldsymbol{v}\Vert_{\boldsymbol{L}^{p}(\Omega)}$ is a norm on $\mathbf{D}(A_{p}^{\alpha})$. This is due to the fact that  (cf. \cite[Theorem 1.15.2, part (e)]{Tri}), the operator $I+A_{p}$ has a bounded inverse and thus for all $\alpha\in\mathbb{C}$ with $\mathrm{Re}\,\alpha>0$, the operator $(I+A_{p})^{\alpha}$ is an isomorphism from $\mathbf{D}(A_{p}^{\alpha})$ to $\boldsymbol{L}^{p}_{\sigma,\tau}(\Omega)$. 

Consider the space 
\begin{equation}\label{w1psigmatau}
\boldsymbol{W}^{1,p}_{\sigma,\tau}(\Omega)\,=\,\big\{ \boldsymbol{v}\in\boldsymbol{W}^{1,p}(\Omega);\,\,\,\mathrm{div}\,\boldsymbol{v}=0\,\,\,\textrm{in}\,\,\Omega,\,\,\,\boldsymbol{v}\cdot\boldsymbol{n}=0\,\,\,\textrm{on}\,\,\Gamma\big\} .
\end{equation}
The following theorem characterize the domain of $A^{1/2}_{p}$.
\begin{theo}\label{DA1/2navierslip}
For all $1<p<\infty$,
$\mathbf{D}(A_{p}^{1/2})\,=\,\boldsymbol{W}^{1,p}_{\sigma,\tau}(\Omega)$ with equivalent norms.
\end{theo}
\begin{proof}
Since the pure imaginary powers of the operator $(I+A_{p})$ are bounded and satisfy estimates \eqref{estimpur1+lap}, one has thanks to \cite[Theorem 1.15.3]{Tri}
\begin{eqnarray*}
\mathbf{D}(A_{p}^{1/2})\,=\,\mathbf{D}((I\,+\,A_{p})^{1/2})&=&\left[\mathbf{D}(I\,+\,A_{p});\boldsymbol{L}^{p}_{\sigma,\tau}(\Omega) \right]_{1/2}\\
&=& \left[\mathbf{D}(A_{p});\boldsymbol{L}^{p}_{\sigma,\tau}(\Omega) \right]_{1/2}.
\end{eqnarray*}
Consider a function $\boldsymbol{u}\in\mathbf{D}(A_{p})$ (see \eqref{c2dpa} for the definition of $\mathbf{D}(A_{p})$) and  set $\boldsymbol{z}=\mathbb{D}(\boldsymbol{u})$ and $\boldsymbol{U}=(\boldsymbol{u},\boldsymbol{z})$. It is clear that when $\boldsymbol{u}\in\mathbf{D}(A_{p})$, the function $\boldsymbol{z}\in\boldsymbol{W}^{1,p}(\Omega)$ and $\boldsymbol{U}\in\boldsymbol{L}^{p}_{\sigma,\tau}(\Omega)\times\boldsymbol{W}^{1,p}(\Omega)$. In addition, when $\boldsymbol{u}\in\boldsymbol{L}^{p}_{\sigma,\tau}(\Omega)$ the function $\boldsymbol{z}\in\boldsymbol{W}^{-1,p}(\Omega)$ and $\boldsymbol{U}\in\boldsymbol{L}^{p}_{\sigma,\tau}(\Omega)\times\boldsymbol{W}^{-1,p}(\Omega)$. Now, let $\boldsymbol{u}\in\mathbf{D}(A_{p}^{1/2})$, then $$\boldsymbol{U}\in\boldsymbol{L}^{p}_{\sigma,\tau}(\Omega)\times\left[\boldsymbol{W}^{1,p}(\Omega)\,;\,\boldsymbol{W}^{-1,p}(\Omega) \right]_{1/2}=\boldsymbol{L}^{p}_{\sigma,\tau}(\Omega)\times\boldsymbol{L}^{p}(\Omega).$$
As a result, $\boldsymbol{u}\in\boldsymbol{L}^{p}(\Omega)$, $\boldsymbol{z}=\mathbb{D}(\boldsymbol{u})\in\boldsymbol{L}^{p}(\Omega)$, $\mathrm{div}\,\boldsymbol{u}=0$ in $\Omega$ and $\boldsymbol{u}\cdot\boldsymbol{n}=0$ on $\Gamma$. Thanks to \cite{REJAIBA}, we know that for all $\boldsymbol{u}\in\boldsymbol{W}^{1,p}_{\sigma,\tau}(\Omega)$, the norm $\Vert\boldsymbol{u}\Vert_{\boldsymbol{W}^{1,p}(\Omega)}$ is equivalent to $\Vert\boldsymbol{u}\Vert_{\boldsymbol{L}^{p}(\Omega)}+\Vert\mathbb{D}(\boldsymbol{u})\Vert_{\boldsymbol{L}^{p}(\Omega)}$. As a result $\boldsymbol{u}\in\boldsymbol{W}^{1,p}_{\sigma,\tau}(\Omega)$ and 
$$\mathbf{D}(A_{p}^{1/2})\hookrightarrow\boldsymbol{W}^{1,p}_{\sigma,\tau}(\Omega).$$

Next we prove the second inclusion. Since $I\,+\,A_{p}$ has a bounded inverse, then for all $1<p<\infty$, the operator $(I\,+\,A_{p})^{1/2}$ is an isomorphism from $\mathbf{D}((I\,+\,A_{p})^{1/2})$ to $\boldsymbol{L}^{p}_{\sigma,\tau}(\Omega)$. This means that for all $\boldsymbol{F}\in\boldsymbol{L}^{p'}_{\sigma,\tau}(\Omega)$ there exists a unique $\boldsymbol{v}\in\mathbf{D}((I\,+\,A_{p'})^{1/2})$ solution of 
\begin{equation}\label{ap'1/2navierslip}
(I\,+\,A_{p'})^{1/2}\boldsymbol{v}\,=\,\boldsymbol{F}.
\end{equation}
Let $\boldsymbol{u}\in\mathbf{D}(A_{p})$ and observe that
\begin{eqnarray}
\Vert(I\,+\,A_{p})^{1/2}\boldsymbol{u}\Vert_{\boldsymbol{L}^{p}(\Omega)}&=&\sup _{\boldsymbol{F}\in\boldsymbol{L}^{p'}_{\sigma,\tau}(\Omega),\,\boldsymbol{F}\neq\boldsymbol{0}}\frac{\Big\vert\langle(I\,+\,A_{p})^{1/2}\boldsymbol{u}\,,\,\boldsymbol{F}\rangle_{\boldsymbol{L}^{p}_{\sigma,\tau}(\Omega)\times\boldsymbol{L}^{p'}_{\sigma,\tau}(\Omega)}\Big\vert}{\Vert\boldsymbol{F}\Vert_{\boldsymbol{L}^{p'}(\Omega)}}\nonumber\\
&=&\sup _{\boldsymbol{v}\in\mathbf{D}(A_{p'}^{1/2}),\,\boldsymbol{v}\neq\boldsymbol{0}}\frac{\Big\vert\langle(I\,+\,A_{p})\boldsymbol{u}\,,\,\boldsymbol{v}\rangle_{\boldsymbol{L}^{p}_{\sigma,\tau}(\Omega)\times\boldsymbol{L}^{p'}_{\sigma,\tau}(\Omega)}\Big\vert}{\Vert(I\,+\,A_{p'})^{1/2}\boldsymbol{v}\Vert_{\boldsymbol{L}^{p'}(\Omega)}}\nonumber\\
&=&\sup _{\boldsymbol{v}\in\mathbf{D}(A_{p'}^{1/2}),\,\boldsymbol{v}\neq\boldsymbol{0}}\frac{\Big\vert\int_{\Omega}\boldsymbol{u}\cdot\overline{\boldsymbol{v}}\,\textrm{d}\,x\,+\,\int_{\Omega}\mathbb{D}(\boldsymbol{u}):\mathbb{D}(\overline{\boldsymbol{v}})\,\textrm{d}\,x\Big\vert}{\Vert(I\,+\,A_{p'})^{1/2}\boldsymbol{v}\Vert_{\boldsymbol{L}^{p'}(\Omega)}}\nonumber\\
&\leq&C(\Omega,p)\,\Vert\boldsymbol{u}\Vert_{\boldsymbol{W}^{1,p}(\Omega)}\label{dpaw1pnavierslip}.
\end{eqnarray}
We recall that $\boldsymbol{v}$ is the unique solution of Problem \eqref{ap'1/2navierslip} and that the adjoint the operator $((I\,+\,A_{p})^{1/2})'$ of $(I\,+\,A_{p})^{1/2}$ is equal to the operator $(I\,+\,A_{p'})^{1/2}$ . We also recall that the dual of  $\boldsymbol{L}^{p}_{\sigma,\tau}(\Omega)$ is equal to $\boldsymbol{L}^{p'}_{\sigma,\tau}(\Omega)$. Using the density of $\mathbf{D}(A_{p})$ in $\boldsymbol{W}^{1,p}_{\sigma,\tau}(\Omega)$, we obtain estimate (\ref{dpaw1pnavierslip}) for all $\boldsymbol{u}\in\boldsymbol{W}^{1,p}_{\sigma,\tau}(\Omega)$ and then
\begin{equation*}
\boldsymbol{W}^{1,p}_{\sigma,\tau}(\Omega)\hookrightarrow\mathbf{D}(A_{p}^{1/2}).
\end{equation*}
\end{proof}

\begin{rmk}
\rm{ In the case where the domain $\Omega$ is not obtained by rotation around a vector $\textbf{\textit{b}}\in \mathbb{R}^{3}$, the Stokes operator $A_{p}$ is invertible with bounded inverse and the following equivalence holds
\begin{equation*}
\forall\,\boldsymbol{u}\in\mathbf{D}(A^{1/2}_{p}),\quad\Vert A_{p}^{1/2}\boldsymbol{u}\Vert_{\boldsymbol{L}^{p}(\Omega)}\,\simeq\,\Vert\mathbb{D}(\boldsymbol{u})\Vert_{\boldsymbol{L}^{p}(\Omega)}.
\end{equation*}
}
\end{rmk}

The following proposition gives us an embeddings of Sobolev type for the domains of fractional powers of the Stokes operator $A_{p}$. 
%
%
%

\begin{prop}\label{SoboembASnavierslip}
for all $1<p<\infty$ and for all $\alpha\in\mathbb{R}$ such that $0<\alpha<3/2p$ the following Sobolev embedding holds 
\begin{equation}\label{sbasnavierslip}
\mathbf{D}(A^{\alpha}_{p})\hookrightarrow\boldsymbol{L}^{q}(\Omega),\qquad\frac{1}{q}=\frac{1}{p}-\frac{2\alpha}{3}.
\end{equation}
Moreover for all $\boldsymbol{u}\in\mathbf{D}(A^{\alpha}_{p})$ the following estimate holds
\begin{equation}\label{estasnavierslip}
\Vert\boldsymbol{u}\Vert_{\boldsymbol{L}^{q}(\Omega)}\,\leq\,C(\Omega,p)\,\Vert (I+A_{p})^{\alpha}\boldsymbol{u}\Vert_{\boldsymbol{L}^{p}(\Omega)}.
\end{equation}
In the particular case where the domain $\Omega$ is not obtained by rotation around a vector $\textbf{\textit{b}}\in \mathbb{R}^{3}$, the following estimate holds
\begin{equation}\label{estasnavierslipinvertible}
\Vert\boldsymbol{u}\Vert_{\boldsymbol{L}^{q}(\Omega)}\,\leq\,C(\Omega,p)\,\Vert A_{p}^{\alpha}\boldsymbol{u}\Vert_{\boldsymbol{L}^{p}(\Omega)}.
\end{equation}
\end{prop}
\begin{proof}
Consider first the case where $0<\alpha<1$ and recall that
\begin{equation*}
\mathbf{D}(A_{p}^{\alpha})\,=\,\mathbf{D}((I\,+\,A_{p})^{\alpha})\,=\,\left[\mathbf{D}(I\,+\,A_{p});\boldsymbol{L}^{p}_{\sigma,\tau}(\Omega) \right]_{\alpha}\,=\, \left[\mathbf{D}(A_{p});\boldsymbol{L}^{p}_{\sigma,\tau}(\Omega) \right]_{\alpha}.
\end{equation*}
The embedding \eqref{sbasnavierslip} is obtained using classical Soblov embedding as in \cite[Theorem 7.57]{Adams}. To extend \eqref{sbasnavierslip} to any real $\alpha$ such that $0<\alpha<3/2p$, we proceed as in the proof of \cite[Corollary 6.11]{ARMA}. This result is similar to the result of  Borchers and Miyakawa \cite{Bor2} who proved the same result for the Stokes operator with Dirichlet boundary conditions in exterior domains for $1<p<3$.  

Estimate \eqref{estasnavierslip} is a direct consequence of \eqref{sbasnavierslip} since the domain $\mathbf{D}(A_{p}^{\alpha})$ is equipped with the graph norm of the operator $(I\,+\,A_{p})^{\alpha}$. 

In the particular case where the domain $\Omega$ is not obtained by rotation around a vector $\textbf{\textit{b}}\in \mathbb{R}^{3}$, the operator $A^{\alpha}_{p}$ is an isomorphism from $\mathbf{D}(A_{p}^{\alpha})$ to $\boldsymbol{L}^{p}_{\sigma,\tau}(\Omega)$. Thus one has estimate \eqref{estasnavierslipinvertible}.
\end{proof}
\section{Applications to the Stokes problem}\label{Applications}
In this section we shall apply the results of Section \ref{Stokes-Fractional} in order to prove maximal  $L^{p}-L^{q}$ for the inhomogeneous Stokes Problem \eqref{StokesProblem}-\eqref{u0}.

Consider first the two problems 
\begin{equation}\label{henpstokesnavierslip}
 \left\{
\begin{array}{cccc}
\frac{\partial\boldsymbol{u}}{\partial t} +A_{p} \boldsymbol{u
}=\boldsymbol{f}, & \mathrm{div}\,\boldsymbol{u}=0&\textrm{in}& 
\Omega\times (0,T), \\
\boldsymbol{u}\cdot\boldsymbol{n}=0, & \left[\mathbb{D}(\textbf{\textit{u}})\textbf{\textit{n}}\right]_{\boldsymbol{\tau}}=\boldsymbol{0}&\textrm{on}&
\Gamma\times (0,T),\\
&\boldsymbol{u}(0)= \boldsymbol{u}_{0} & \textrm{in} &
\Omega
\end{array}
\right.
\end{equation}
and
\begin{equation}\label{henpnavierslip}
 \left\{
\begin{array}{cccc}
\frac{\partial\boldsymbol{u}}{\partial t} - \Delta \boldsymbol{u
}+\nabla\pi=\boldsymbol{f}, & \mathrm{div}\,\boldsymbol{u}=0&\textrm{in}& 
\Omega\times (0,T), \\
\boldsymbol{u}\cdot\boldsymbol{n}=0, & \left[\mathbb{D}(\textbf{\textit{u}})\textbf{\textit{n}}\right]_{\boldsymbol{\tau}}=\boldsymbol{0}&\textrm{on}&
\Gamma\times (0,T),\\
&\boldsymbol{u}(0)= \boldsymbol{u}_{0} & \textrm{in} &
\Omega,
\end{array}
\right.
\end{equation}
where $\boldsymbol{u}_{0}\in\boldsymbol{L}^{p}_{\sigma,\tau}(\Omega),\,$ $\boldsymbol{f}\in L^q(0,T;\,\boldsymbol{L}^{p}_{\sigma,\tau}(\Omega)$ and $1<p,q<\infty$.
Notice that a function $\boldsymbol{u}\in
C(]0,\,+\infty[,\,\mathbf{D}(A_{p}))\cap
C^{1}(]0,\,+\infty[,\,\boldsymbol{L}^{p}_{\sigma,\tau}(\Omega))$ solves \eqref{henpstokesnavierslip} if and only if there exists a function $\pi\in C(\left] 0,\infty\right[;\,W^{1,p}(\Omega)/\mathbb{R} )$ such that $(\boldsymbol{u},\,\pi)$ solves \eqref{henpnavierslip}. Indeed, let $\boldsymbol{u}$ be a solution to \eqref{henpstokesnavierslip}. Thus $A_{p}\boldsymbol{u}=-P\Delta\boldsymbol{u}=\boldsymbol{f}-\frac{\partial\boldsymbol{u}}{\partial t}$, where $P$ is the Helmholtz projection defined by \eqref{helmholtzproj}-\eqref{wn.1}. Since $(\boldsymbol{u},\,\boldsymbol{f}-\frac{\partial\boldsymbol{u}}{\partial t})\in\mathbf{D}(A_{p})\times\boldsymbol{L}^{p}_{\sigma,\tau}(\Omega)$, then due to \cite[Theorem 4.1]{REJAIBA} there exists $\pi\in W^{1,p}(\Omega)/\mathbb{R}$ such that $A_{p}\boldsymbol{u}=-\Delta\boldsymbol{u}+\nabla\pi=-\boldsymbol{f}-\frac{\partial\boldsymbol{u}}{\partial t}$. Moreover we have the estimate
\begin{equation*}
\Vert\boldsymbol{u}\Vert_{\boldsymbol{W}^{2,p}(\Omega)/\boldsymbol{\mathcal{T}}^{p}(\Omega)}+\Vert\pi\Vert_{W^{1,p}(\Omega)/\mathbb{R}}\leq C(\Omega,p)\Big\Vert\frac{\partial\boldsymbol{u}}{\partial t}\Big\Vert_{\boldsymbol{L}^{p}(\Omega)},
\end{equation*}
where $\boldsymbol{\mathcal{T}}^{p}(\Omega)$ is the Kernel of the Stokes operator with Navier-slip boundary condition described above. This means that that the mapping $\boldsymbol{f}-\frac{\partial\boldsymbol{u}}{\partial t}\mapsto\pi$ is continuous from $\boldsymbol{L}^{p}_{\sigma,\tau}(\Omega)$ to $W^{1,p}(\Omega)$. As a result, $\pi\in C(\left] 0,\infty\right[ ;\,W^{1,p}(\Omega)/\mathbb{R})$ and $(\boldsymbol{u},\pi)$ solves \eqref{henpnavierslip}. Conversely, let $(\boldsymbol{u},\pi)$ be a solution of \eqref{henpnavierslip}. Applying the Helmholtz-projection $P$ to the first equation of Problem \eqref{henpnavierslip}, one gets directly that $\boldsymbol{u}$ solves \eqref{henpstokesnavierslip}.

\medskip

For the homogeneous problem (\textit{i.e.} $\boldsymbol{f}=\boldsymbol{0}$), the analyticity of the semi-group give us a unique solutions satisfying all the regularity desired. As stated in \cite{Rejaiba}, when the initial data $\boldsymbol{u}_{0}\in\boldsymbol{L}^{p}_{\sigma,\tau}(\Omega)$ and when $\boldsymbol{f}=\boldsymbol{0}$,
the Problem \eqref{StokesProblem}-\eqref{u0} has a unique solution $(\boldsymbol{u},\pi)$ satisfying
\begin{equation*}
\boldsymbol{u}\in
C([0,\,+\infty[,\,\boldsymbol{L}^{p}_{\sigma,\tau}(\Omega))\cap
C(]0,\,+\infty[,\,\mathbf{D}(A_{p}))\cap
C^{1}(]0,\,+\infty[,\,\boldsymbol{L}^{p}_{\sigma,\tau}(\Omega)),
\end{equation*}
\begin{equation*}
\boldsymbol{u}\in C^{k}(]0,\,+\infty[,\,\mathbf{D}(A_{p}^{\ell})),\qquad
\forall\,k\in\mathbb{N},\,\,\forall\,\ell\in\mathbb{N^{\ast}},
\end{equation*}
\begin{equation*}
\pi\in C(\left] 0,\infty\right[;\,W^{1,p}(\Omega)/\mathbb{R} ).
\end{equation*}
Moreover the following estimates hold
\begin{equation}\label{estanalysis1}
\|\boldsymbol{u}(t)\|_{\boldsymbol{L}^{p}(\Omega)}\leq\,C(\Omega,p)\,\|\boldsymbol{u}_{0}\|_{\boldsymbol{L}^{p}(\Omega)},
\end{equation}
\begin{equation*}
\Big\|\frac{\partial\boldsymbol{u}(t)}{\partial t}\Big\|_{\boldsymbol{L}^{p}(\Omega)}\leq\frac{C(\Omega,p)}{t}\,\|\boldsymbol{u}_{0}\|_{\boldsymbol{L}^{p}(\Omega)},
\end{equation*}
\begin{equation}\label{estanal2}
\Vert\mathbb{D}(\textbf{\textit{u}})\Vert_{\boldsymbol{L}^{p}(\Omega)}\,\leq\,\frac{C(\Omega,p)}{\sqrt{t}}\,\Vert\boldsymbol{u}_{0}\Vert_{\boldsymbol{L}^{p}(\Omega)}.
\end{equation}

\begin{rmk}
\rm{In the case where the domain $\Omega$ is not obtained by rotation around a vector $b\in\mathbb{R}^{3}$, the Stokes semi-group decays exponentially and we can extend estimates \eqref{estanalysis1}-\eqref{estanal2} to the following $L^p-L^q$ estimates. Exactly, for every $p$ and $q$ such that  $1<p\leq q<\infty$, for every $\boldsymbol{u}_{0}\in\boldsymbol{L}^{p}_{\sigma,\tau}(\Omega)$ and $\boldsymbol{f}=\boldsymbol{0}$, there exists a constant $\delta>0$ such that the unique solution $\boldsymbol{u}(t)$ to Problem \eqref{StokesProblem}-\eqref{u0} belongs to $\boldsymbol{L}^{q}_{\sigma,\tau}(\Omega)$ and satisfies :
 \begin{equation}\label{estlplqutnavierslip}
 \Vert\boldsymbol{u}(t)\Vert_{\boldsymbol{L}^{q}(\Omega)}\,\leq\,C\,e^{-\delta t}t^{-3/2(1/p-1/q)}\Vert \boldsymbol{u}_{0}\Vert_{\boldsymbol{L}^{p}(\Omega)},
 \end{equation}
 \begin{equation*}\label{estlplqdunavierslip}
  \Vert \mathbb{D}(\boldsymbol{u}(t))\Vert_{\boldsymbol{L}^{q}(\Omega)}\,\leq\,C\,e^{-\delta t}t^{-1/2}\,t^{-3/2(1/p-1/q)}\Vert \boldsymbol{u}_{0}\Vert_{\boldsymbol{L}^{p}(\Omega)},
 \end{equation*}
 \begin{equation}\label{estlplqlaputnavierslip}
 \forall\,m,n\in\mathbb{N},\qquad \Big\Vert\frac{\partial^{m}}{\partial t^{m}}A^{n}_{p}\boldsymbol{u}(t)\Big\Vert_{\boldsymbol{L}^{q}(\Omega)}\,\leq\,C\,e^{-\delta t}t^{-(m+n)}\,t^{-3/2(1/p-1/q)}\Vert \boldsymbol{u}_{0}\Vert_{\boldsymbol{L}^{p}(\Omega)}.
 \end{equation}}
Estimate \eqref{estlplqutnavierslip}-\eqref{estlplqlaputnavierslip} are obtained using the embedding of Sobolev type \eqref{sbasnavierslip}, estimate \eqref{estasnavierslip} and the fact that for all $\alpha\in\mathbb{R}$ we have
 \begin{equation*}
\Vert A^{\alpha}\boldsymbol{u}(t)\Vert_{\boldsymbol{L}^{p}(\Omega)}=\Vert A^{\alpha}\,e^{-t\,A_{p}}\boldsymbol{u}_{0}\Vert_{\boldsymbol{L}^{p}(\Omega)}\,\leq\, C\,e^{-\delta t}t^{-\alpha}\Vert\boldsymbol{u}_{0}\Vert_{\boldsymbol{L}^{p}(\Omega)}.
 \end{equation*}
\end{rmk}

\medskip

Consider now the non-homogeneous case, where $\boldsymbol{u}_{0}=\boldsymbol{0}$ and $\boldsymbol{f}\in L^{q}(0,T;\,\boldsymbol{L}^{p}_{\sigma,\tau}(\Omega))$, with $1<p,q<\infty$ and $0<T\leq\infty$. It is well known (cf. \cite{Pa}), that for such $\boldsymbol{f}$, Problem \eqref{StokesProblem}-\eqref{u0} has a unique solution $\boldsymbol{u}\in C(0,T;\,\boldsymbol{L}^{p}_{\sigma,\tau}(\Omega))$. It is also known that for such  $\boldsymbol{f}$ the analyticity of the Stokes semi-group is not enough to obtain a unique solution $(\boldsymbol{u},\pi)$ satisfying the following maximal $L^p-L^q$ regularity :
\begin{equation*}
\boldsymbol{u}\in L^{q}(0,T;\,\mathbf{D}(A_{p})),\qquad \frac{\partial \boldsymbol{u}}{\partial t}\in L^{q}(0,T;\,L^{p}_{\sigma,\tau}(\Omega)),
\end{equation*}
\begin{equation*}
\pi\in L^{q}(0,T;\,W^{1,p}(\Omega)/\mathbb{R}).
\end{equation*}

In what follows we prove maximal $L^p-L^q$ regularity of the solution to the Stokes Problem \eqref{StokesProblem}-\eqref{u0} using the boundedness of the pure imaginary powers of the operator $I+A_{p}$ and \cite[Theorem 2.1]{GiGa4}.
\begin{theo}[Strong solution to the Stokes Problem]\label{Exisinhnsplpnavierslip}
Let $0<T\leq\infty$, $1<p,q<\infty$, $\boldsymbol{f}\in\boldsymbol{L}^{q}(0,T;\boldsymbol{L}^{p}_{\sigma,\tau}(\Omega))$ and $\boldsymbol{u}_{0}=\boldsymbol{0}$. The Problem \eqref{StokesProblem}-\eqref{u0} has a unique solution $(\boldsymbol{u},\pi)$ such that
\begin{equation}\label{reglplq1}
\boldsymbol{u}\in L^{q}(0,T_{0};\,\boldsymbol{W}^{2,p}(\Omega)),\,\,\,\,  T_{0}\leq T\,\,\,\textrm{if}\,\,\,T<\infty\,\,\,\, \mathrm{and }\,\,\,\,T_{0}<T\,\,\,\textrm{if}\,\,\,T=\infty,
\end{equation}
\begin{equation*}
\pi\in L^{q}(0,T;\,W^{1,p}(\Omega)/\mathbb{R}),\qquad \frac{\partial\boldsymbol{u}}{\partial t}\in L^{q}(0,T;\,\boldsymbol{L}^{p}(\Omega))
\end{equation*}
and
\begin{multline}\label{estlplq1}
\int_{0}^{T}\Big\Vert\frac{\partial\boldsymbol{u}}{\partial t}\Big\Vert^{q}_{\boldsymbol{L}^{p}(\Omega)}\,\mathrm{d}\,t\,+\,\int_{0}^{T}\Vert A_{p}\boldsymbol{u}(t)\Vert^{q}_{\boldsymbol{L}^{p}(\Omega)}\,\mathrm{d}\,t\,+\,\int_{0}^{T}\Vert\pi(t)\Vert^{q}_{W^{1,p}(\Omega)/\mathbb{R}}\,\mathrm{d}\,t\\
\leq\,C(p,q,\Omega)\,\int_{0}^{T}\Vert\boldsymbol{f}(t)\Vert^{q}_{\boldsymbol{L}^{p}(\Omega)}\,\mathrm{d}\,t.
\end{multline}
\end{theo}
\begin{proof}
As stated above Problem \eqref{StokesProblem}-\eqref{u0}  has a unique solution $\boldsymbol{u}\in C(0,T;\,\boldsymbol{L}^{p}_{\sigma,\tau}(\Omega))$. Let us prove that this solution satisfy the maximal $L^{p}-L^{q}$ regularity \eqref{reglplq1}. Indeed, let $\mu>0$ and set $\boldsymbol{u}_{\mu}(t)=e^{-\frac{1}{\mu^{2}}t}\boldsymbol{u}(t)$, the function $\boldsymbol{u}_{\mu}(t)$ is a solution to the following problem
\begin{equation}\label{inhenspv(t)}
 \left\{
\begin{array}{cccc}
\frac{\partial\boldsymbol{u}_{\mu}}{\partial t} + (\frac{1}{\mu^{2}}I\,+\,A_{p})\boldsymbol{u}_{\mu}(t)=e^{-\frac{1}{\mu^{2}}t}\boldsymbol{f},& 
\mathrm{div}\,\boldsymbol{u}_{\mu}(t)= 0 &\mathrm{in}&\Omega\times (0,T), \\
\boldsymbol{u}_{\mu}(t)\cdot\boldsymbol{n}=0,& 
[\mathbb{D}(\boldsymbol{u}_{\mu}(t))\boldsymbol{n}]_{\tau} = \boldsymbol{0} &\mathrm{on} & \Gamma\times (0,T), \\
&\boldsymbol{u}_{\mu}(0)=\boldsymbol{u}(0)=\boldsymbol{0} &\mathrm{in}&
\Omega.
\end{array}
\right.
\end{equation}
Since the pure imaginary powers of the operator $(\frac{1}{\mu^{2}}I\,+\,A_{p})$ are bounded in $\boldsymbol{L}^{p}_{\sigma,\tau}(\Omega)$  (see Theorem \ref{Lapimpowerproposition}) and since for all $1<p<\infty$, $\boldsymbol{L}^{p}_{\sigma,\tau}(\Omega)$ is $\zeta$-convex, we can apply the result of \cite[Theorem 2.1]{GiGa4} to the operator $(\frac{1}{\mu^{2}}I\,+\,A_{p})$. Thus, the solution $\boldsymbol{u}_{\mu}(t)$ to the Problem \eqref{inhenspv(t)} satisfies the following maximal $L^{p}-L^{q}$ regularity
\begin{equation}\label{reglplqlap1vt}
\boldsymbol{u}_{\mu}\in L^{q}(0,T_{0};\,\mathbf{D}(A_{p}))\cap W^{1,q}(0,T;\,\boldsymbol{L}^{p}_{\sigma,\tau}(\Omega)),
\end{equation}
with $T_{0}\leq T\,$ if $T<\infty\,$ and $T_{0}<T\,$ if $T=\infty$. Furthermore $\boldsymbol{u}_{\mu}(t)$ satisfies the following estimate
\begin{multline}\label{estlplqlapvt}
\int_{0}^{T}\Big\Vert\frac{\partial\boldsymbol{u}_{\mu}}{\partial t}\Big\Vert^{q}_{\boldsymbol{L}^{p}(\Omega)}\,\mathrm{d}\,t\,+\,\int_{0}^{T}\Big\Vert\Big(\frac{1}{\mu^{2}}I\,+\,A_{p}\Big)\boldsymbol{u}_{\mu}(t)\Big\Vert^{q}_{\boldsymbol{L}^{p}(\Omega)}\,\mathrm{d}\,t\,\leq\\
C(p,q,\Omega)\,\int_{0}^{T}\Vert e^{-\frac{1}{\mu^{2}}t}\boldsymbol{f}(t)\Vert^{q}_{\boldsymbol{L}^{p}(\Omega)}\,\mathrm{d}\,t\leq C(p,q,\Omega)\,\int_{0}^{T}\Vert\boldsymbol{f}(t)\Vert^{q}_{\boldsymbol{L}^{p}(\Omega)}\,\mathrm{d}\,t,
\end{multline}
where the constant $C(p,q,\Omega)$ is independent of $\mu$.

Since the solution $\boldsymbol{u}$ to Problem \eqref{StokesProblem}-\eqref{u0} can be written in the form $\boldsymbol{u}(t)=e^{\frac{1}{\mu^{2}}}\boldsymbol{u}_{\mu}(t),\,$ we deduce from \eqref{reglplqlap1vt} that $\boldsymbol{u}$ satisfies
$$\boldsymbol{u}\in L^{q}(0,T_{0};\,\mathbf{D}(A_{p}))\cap W^{1,q}(0,T;\,\boldsymbol{L}^{p}_{\sigma,\tau}(\Omega)),$$
with $T_{0}\leq T\,$ if $T<\infty\,$ and $T_{0}<T\,$ if $T=\infty$.  

Using now the fact that $A_{p}\boldsymbol{u}=-\Delta\boldsymbol{u}+\nabla\pi=\boldsymbol{f}-\frac{\partial\boldsymbol{u}}{\partial t}$, one has thanks to \cite[Theorem 4.1]{REJAIBA}
\begin{equation}\label{pistrong}
\Vert\boldsymbol{u}(t)\Vert_{\boldsymbol{W}^{2,p}(\Omega)/\boldsymbol{\mathcal{T}}^{p}(\Omega)}\,+\,\Vert\pi\Vert_{W^{1,p}(\Omega)/\mathbb{R}}\,\leq\,C\,\Big(\Big\Vert\frac{\partial\boldsymbol{u}}{\partial t}\Big\Vert_{\boldsymbol{L}^{p}(\Omega)}+\Vert\boldsymbol{f}\Vert_{\boldsymbol{L}^{p}(\Omega)}\Big).
\end{equation}
As a result we deduce that $\pi\in L^{q}(0,T;\, W^{1,p}(\Omega)/\mathbb{R})$.

 It remains to prove estimate \eqref{estlplq1}. We recall first the following equivalence of norms
 $$\forall\,\boldsymbol{v}\in\mathbf{D}(A_{p}),\quad\Vert \boldsymbol{v}\Vert_{\boldsymbol{W}^{2,p}(\Omega)}\,\simeq\, \Vert \boldsymbol{v}\Vert_{\mathbf{D}(A_{p})} \,\simeq\, \Big\Vert\Big(\frac{1}{\mu^{2}}I\,+\,A_{p}\Big)\boldsymbol{v}\Big\Vert_{\boldsymbol{L}^{p}(\Omega)}. $$
Then, substituting in \eqref{estlplqlapvt} we have
 \begin{multline}\label{estlplqlapvt2}
\int_{0}^{T}\Big\Vert\frac{\partial\boldsymbol{u}_{\mu}}{\partial t}\Big\Vert^{q}_{\boldsymbol{L}^{p}(\Omega)}\,\mathrm{d}\,t\,+\,\int_{0}^{T}\Vert\boldsymbol{u}_{\mu}(t)\Vert^{q}_{\mathbf{D}(A_{p})}\,\mathrm{d}\,t\,\leq
C(p,q,\Omega)\,\int_{0}^{T}\Vert\boldsymbol{f}(t)\Vert^{q}_{\boldsymbol{L}^{p}(\Omega)}\,\mathrm{d}\,t,
\end{multline}
where the constant $C(p,q,\Omega)$ is independent of $\mu$.
Next observe that
\begin{equation*}
\boldsymbol{u}_{\mu}\longrightarrow\boldsymbol{u}\quad\mathrm{in}\,\,\, L^{q}(0,T_{0};\,\mathbf{D}(A_{p}))\cap W^{1,q}(0,T_{0};\,\boldsymbol{L}^{p}_{\sigma,\tau}(\Omega)).
\end{equation*}
Using then the dominated convergence theorem and passing to the limit as $\mu$ tends to infinity in \eqref{estlplqlapvt2} we obtain
\begin{equation*}
\int_{0}^{T}\Big\Vert\frac{\partial\boldsymbol{u}}{\partial t}\Big\Vert^{q}_{\boldsymbol{L}^{p}(\Omega)}\,\mathrm{d}\,t\,+\,\int_{0}^{T}\Vert\boldsymbol{u}(t)\Vert^{q}_{\mathbf{D}(A_{p})}\,\mathrm{d}\,t\,\leq\,
C(p,q,\Omega)\,\int_{0}^{T}\Vert\boldsymbol{f}(t)\Vert^{q}_{\boldsymbol{L}^{p}(\Omega)}\,\mathrm{d}\,t.
\end{equation*}
Finally using \eqref{pistrong} and the fact that $\Vert \boldsymbol{u}\Vert_{\mathbf{D}(A_{p})}$ is equivalent to $\Vert \boldsymbol{u}\Vert_{\boldsymbol{L}^{p}(\Omega)}\,+\,\Vert A_{p}\boldsymbol{u}\Vert_{\boldsymbol{L}^{p}(\Omega)}$,  estimate \eqref{estlplq1} follows directly.
\end{proof}

The boundedness of the pure imaginary powers of the operators $\lambda I+B_{p}$ and $\lambda I+C_{p},$ with $\lambda>0$ on the spaces $[\boldsymbol{H}^{p'}_{0}(\mathrm{div},\Omega)]'_{\sigma,\tau}$ and $[\boldsymbol{T}^{p'}(\Omega)]'_{\sigma,\tau}$ respectively, (see Theorem \ref{PureimgI+Bp+Cp}), allows us to obtain weak and very weak solutions to Problem \eqref{StokesProblem}-\eqref{u0}. Indeed, using \cite[Proposition 2.16]{ARMA} we know that the spaces $[\boldsymbol{H}^{p'}_{0}(\mathrm{div},\Omega)]'_{\sigma,\tau}$ and $[\boldsymbol{T}^{p'}(\Omega)]'_{\sigma,\tau}$ are $\zeta$-convex Banach spaces. As a result, proceeding as in the proof of Theorem \ref{Exisinhnsplpnavierslip} we obtain the following two theorems.

\begin{theo}[Weak solution to the Stokes Problem]\label{Existinhsphdivnavierslip}
Let $1<p,q<\infty$, $\boldsymbol{u}_{0}=0$ and let $\boldsymbol{f}\in L^{q}(0,T;\,[\boldsymbol{H}^{p'}_{0}(\mathrm{div},\Omega)]'_{\sigma,\tau})$, $0<T\leq\infty$. The Problem  \eqref{StokesProblem}-\eqref{u0} has a unique solution $(\boldsymbol{u},\pi)$ satisfying
\begin{equation*}
\boldsymbol{u}\in L^{q}(0,T_{0};\,\,\boldsymbol{W}^{1,p}(\Omega)),\,\,\,\,  T_{0}\leq T\,\,\,\textrm{if}\,\,\,T<\infty\,\,\,\, \mathrm{and }\,\,\,\,T_{0}<T\,\,\,\textrm{if}\,\,\,T=\infty,
\end{equation*}
\begin{equation*}
\pi\in L^{q}(0,T;\,\,L^{p}(\Omega)/\mathbb{R}),\qquad\frac{\partial\boldsymbol{u}}{\partial t}\in L^{q}(0,T;\,\in[\boldsymbol{H}^{p'}_{0}(\mathrm{div}\Omega)]'_{\sigma,T})
\end{equation*}
and
\begin{multline}\label{estlplaweak}
\int_{0}^{T}\Big\Vert\frac{\partial\boldsymbol{u}}{\partial t}\Big\Vert^{q}_{[\boldsymbol{H}^{p'}_{0}(\mathrm{div}\Omega)]'}\,\mathrm{d}\,t\,+\,\int_{0}^{T}\Vert B_{p}\boldsymbol{u}(t)\Vert^{q}_{[\boldsymbol{H}^{p'}_{0}(\mathrm{div}\Omega)]'}\,\mathrm{d}\,t\,+\,\int_{0}^{T}\Vert\pi(t)\Vert^{q}_{L^{p}(\Omega)/\mathbb{R}}\,\mathrm{d}\,t\\
\leq\,C(p,q,\Omega)\,\int_{0}^{T}\Vert\boldsymbol{f}(t)\Vert^{q}_{[\boldsymbol{H}^{p'}_{0}(\mathrm{div}\Omega)]'}\,\mathrm{d}\,t.
\end{multline}
 \end{theo}
 \begin{proof}
 Proceeding in the same way as in the proof of Theorem \ref{Exisinhnsplpnavierslip}, using the boundedness of the pure imaginary powers of the operators $\frac{1}{\mu^2}\, I+B_{p},\,$ $\mu>0$, on $[\boldsymbol{H}^{p'}_{0}(\mathrm{div},\Omega)]'_{\sigma,\tau}$ and the change of variable $\boldsymbol{u}_{\mu}(t)=e^{-\frac{1}{\mu^{2}}t}\boldsymbol{u}(t)$, we obtain that Problem \eqref{StokesProblem}-\eqref{u0} has a unique solution satisfying $$\boldsymbol{u}\in L^{q}(0,T_{0};\,\,\boldsymbol{W}^{1,p}(\Omega))\cap W^{1,q}(0,T;\,[\boldsymbol{H}^{p'}_{0}(\mathrm{div}\Omega)]'_{\sigma,T}),$$
 with $T_{0}\leq T\,$ if $T<\infty\,$ and $T_{0}<T\,$ if $T=\infty$.
 Next, using that, $B_{p}\boldsymbol{u}=-\Delta\boldsymbol{u}+\nabla\pi=\boldsymbol{f}-\frac{\partial\boldsymbol{u}}{\partial t}$, one has thanks to \cite[Theorems 3.7, 3.9]{REJAIBA}
\begin{equation}\label{piweak}
\Vert\boldsymbol{u}(t)\Vert_{\boldsymbol{W}^{1,p}(\Omega)/\boldsymbol{\mathcal{T}}^{p}(\Omega)}\,+\,\Vert\pi\Vert_{L^{p}(\Omega)/\mathbb{R}}\,\leq\,\Big\Vert\frac{\partial\boldsymbol{u}}{\partial t}\Big\Vert_{[\boldsymbol{H}^{p'}_{0}(\mathrm{div}\Omega)]'}+\Vert\boldsymbol{f}\Vert_{[\boldsymbol{H}^{p'}_{0}(\mathrm{div}\Omega)]'}.
\end{equation}
As a result we deduce that $\pi\in L^{q}(0,T;\, L^{p}(\Omega)/\mathbb{R})$ and estimate \eqref{estlplaweak} follows.
 \end{proof}

\medskip

\begin{theo}[Very weak solutions to the Stokes Problem]\label{Existinhsptpnavierslip} 
Let $0<T\leq\infty$, $1<p,q<\infty$,  $\boldsymbol{u}_{0}=0$ and $\boldsymbol{f}\in L^{q}(0,T;\,[\boldsymbol{T}^{p'}(\Omega)]'_{\sigma,\tau})$. Then the evolutionary Stokes Problem \eqref{StokesProblem}-\eqref{u0} has a unique solution $(\boldsymbol{u},\pi)$ satisfying
\begin{equation*}
\boldsymbol{u}\in L^{q}(0,T_{0};\,\boldsymbol{L}^{p}(\Omega)),\,\,\,\,  T_{0}\leq T\,\,\,\textrm{if}\,\,\,T<\infty\,\,\,\, \mathrm{and }\,\,\,\,T_{0}<T\,\,\,\textrm{if}\,\,\,T=\infty,
\end{equation*}  
\begin{equation*}
\pi\in L^{q}(0,T;\,\,W^{-1,p}(\Omega)/\mathbb{R}),\qquad\frac{\partial\boldsymbol{u}}{\partial t}\in L^{q}(0,T;\,\in[\boldsymbol{T}^{p'}(\Omega)]'_{\sigma,\tau})
\end{equation*}
and
\begin{multline*}
\int_{0}^{T}\Big\Vert\frac{\partial\boldsymbol{u}}{\partial t}\Big\Vert^{q}_{[\boldsymbol{T}^{p'}(\Omega)]'}\,\mathrm{d}\,t\,+\,\int_{0}^{T}\Vert C_{p}\boldsymbol{u}(t)\Vert^{q}_{[\boldsymbol{T}^{p'}(\Omega)]'}\,\mathrm{d}\,t\,+\,\int_{0}^{T}\Vert\pi(t)\Vert^{q}_{W^{-1,p}(\Omega)/\mathbb{R}}\,\mathrm{d}\,t\\
\leq\,C(p,q,\Omega)\,\int_{0}^{T}\Vert\boldsymbol{f}(t)\Vert^{q}_{[\boldsymbol{T}^{p'}(\Omega)]'}\,\mathrm{d}\,t.
\end{multline*}
 \end{theo}
  \begin{proof}
  The proof is similar to the proof of Theorem \ref{Exisinhnsplpnavierslip}. First we use the boundedness of the pure imaginary powers of $\frac{1}{\mu^2}\, I+C_{p},$ with $\mu>0$ on the space $[\boldsymbol{T}^{p'}(\Omega)]'_{\sigma,\tau}$ and the change of variable $\boldsymbol{u}_{\mu}(t)=e^{-\frac{1}{\mu^{2}}t}\boldsymbol{u}(t)$.  We obtain then a unique solution to Problem \eqref{StokesProblem}-\eqref{u0} satisfying $$\boldsymbol{u}\in L^{q}(0,T_{0};\,\,\boldsymbol{L}^{p}(\Omega))\cap W^{1,q}(0,T;\,[\boldsymbol{T}^{p'}(\Omega)]'_{\sigma,\tau}),$$
 with $T_{0}\leq T\,$ if $T<\infty\,$ and $T_{0}<T\,$ if $T=\infty$.
 Next, using the fact $C_{p}\boldsymbol{u}=-\Delta\boldsymbol{u}+\nabla\pi=\boldsymbol{f}-\frac{\partial\boldsymbol{u}}{\partial t}$, one has thanks to \cite[Theorem 5.5]{REJAIBA}
\begin{equation}\label{piweak}
\Vert\boldsymbol{u}(t)\Vert_{\boldsymbol{L}^{p}(\Omega)/\boldsymbol{\mathcal{T}}^{p}(\Omega)}\,+\,\Vert\pi\Vert_{W^{-1,p}(\Omega)/\mathbb{R}}\,\leq\,\Big\Vert\frac{\partial\boldsymbol{u}}{\partial t}\Big\Vert_{[\boldsymbol{T}^{p'}(\Omega)]'}+\Vert\boldsymbol{f}\Vert_{[\boldsymbol{T}^{p'}(\Omega)]'}.
\end{equation}
As a result we deduce that $\pi\in L^{q}(0,T;\, W^{-1,p}(\Omega)/\mathbb{R})$ and estimate \eqref{estlplaweak} follows.
 \end{proof}

\section{Appendix}\label{appendix}
In this section we prove estimate \eqref{estnuDeltaM*} in the case where the domain $\Omega$ is not necessarily star shaped with respect to one of its points.

First let $\boldsymbol{u}\in\boldsymbol{L}^{p}(\Omega)$ be fixed and consider the mapping
\begin{eqnarray*}
\widetilde{A}_{p}\boldsymbol{u}&:&\quad \boldsymbol{W}\longrightarrow \mathbb{C}\\
& &\quad \boldsymbol{v} \longrightarrow -\int_{\Omega}\boldsymbol{u}\cdot\Delta\overline{\boldsymbol{v}}\,\mathrm{d}\,x,
\end{eqnarray*}
where
\begin{equation*}
\boldsymbol{W}=\boldsymbol{W}^{1,p'}_{\sigma,\tau}(\Omega)\cap\boldsymbol{W}^{2,p'}(\Omega).
\end{equation*}
We recall that for all $1<p<\infty$, the space $\boldsymbol{W}^{1,p}_{\sigma,\tau}(\Omega)$ is defined by \eqref{w1psigmatau}.
It is clear that $\widetilde{A}_{p}\in\mathcal{L}(\boldsymbol{L}^{p}(\Omega),\,\boldsymbol{W}')$ and thanks to de Rham's Lemma there exists $\pi\in W^{-1,p}(\Omega)$ (the dual of $W^{1,p'}_{0}(\Omega)$) such that $$\widetilde{A}_{p}\boldsymbol{u}+\Delta\boldsymbol{u}=\nabla\pi\qquad\mathrm{in}\quad\Omega.$$

Now suppose that $\widetilde{A}_{p}\boldsymbol{u}\in\boldsymbol{L}^{p}(\Omega)$. Since $\Delta\boldsymbol{u}=-\widetilde{A}_{p}\boldsymbol{u}+\nabla\pi\,$ in $\Omega$, then using \cite[Lemma 5.4]{REJAIBA}, we deduce that $\,[\mathbb{D}(\boldsymbol{u})\boldsymbol{n}]_{\tau}\in\boldsymbol{W}^{-1-1/p,p}(\Gamma).$ Furthermore, if we suppose that $\mathrm{div}\,\boldsymbol{u}=0\,$ in $\Omega,\,$ $\boldsymbol{u}\cdot\boldsymbol{n}=0\,$ on $\Gamma\,$ and $[\mathbb{D}(\boldsymbol{u})\boldsymbol{n}]_{\tau}=\boldsymbol{0}\,$ on $\,\Gamma$, then $(\boldsymbol{u},\,\pi)\in\boldsymbol{L}^{p}_{\sigma,\tau}(\Omega)\times W^{-1,p}(\Omega)$ is a solution to the following problem
 \begin{equation*}
\left\{
\begin{array}{r@{~}c@{~}l}
 - \Delta \boldsymbol{u}\,+\,\nabla\pi =\widetilde{A}_{p}\boldsymbol{u}, &&\quad\mathrm{div}\,\boldsymbol{u} = 0 \,\,\, \qquad\qquad \mathrm{in} \,\,\, \Omega, \\
\boldsymbol{u}\cdot \boldsymbol{n} = 0, &&\quad [\mathbb{D}(\boldsymbol{u})\boldsymbol{n}]_{\tau}=\boldsymbol{0}\qquad
\,\,\mathrm{on}\,\,\, \Gamma.
\end{array}
\right.
\end{equation*}
 Using then the regularity of the Stokes Problem \cite[Theorem 4.1]{REJAIBA} one has $(\boldsymbol{u},\,\pi)\in\boldsymbol{W}^{2,p}(\Omega)\times W^{1,p}(\Omega)$ since $\Omega$ is of class $C^{2,1}.$ 

\vspace{0.25cm}
Summarizing, the operator $\widetilde{A}_{p}:\mathbf{D}(\widetilde{A}_{p})\subset\boldsymbol{L}^{p}(\Omega)\longmapsto\boldsymbol{L}^{p}(\Omega)$ is a closed linear operator with
\begin{equation*}
\mathbf{D}(\widetilde{A}_{p})=\big\{ \boldsymbol{u} \in
\boldsymbol{W}^{2,p}(\Omega);\,\,\mathrm{div}\,\boldsymbol{u}=
0\,\, \mathrm{in}\,\,\Omega,
\,\,\boldsymbol{u}\cdot\boldsymbol{n}=0,\,\,\,[\mathbb{D}(\boldsymbol{u})\boldsymbol{n}]_{\tau}=\boldsymbol{0}
\,\,\mathrm{on}\,\,\Gamma \big\},
\end{equation*} 
\begin{equation}\label{Aptilde1}
\forall\,\,\boldsymbol{u}\in\mathbf{D}(\widetilde{A}_{p}),\qquad
\widetilde{A}_{p}\boldsymbol{u}\,=\,-\Delta\boldsymbol{u}\,+\boldsymbol{\mathrm{grad}}\,\pi.
\end{equation}
Observe that for all $\boldsymbol{u}\in\mathbf{D}(\widetilde{A}_{p})$ and for all $\boldsymbol{v}\in\boldsymbol{W}^{1,p'}_{\sigma,\tau}(\Omega)$ one has \begin{equation*}
\int_{\Omega}\widetilde{A}_{p}\,\boldsymbol{u}\cdot\overline{\boldsymbol{v}}\,\mathrm{d}\,x\,=\,\int_{\Omega}\mathbb{D}(\boldsymbol{u})\boldsymbol{:}\mathbb{D}(\overline{\boldsymbol{v}})\,\textrm{d}\,x.
\end{equation*}
 Notice also that, if $\,\widetilde{A}_{p}\boldsymbol{u}\in\boldsymbol{L}^{p}_{\sigma,\tau}(\Omega),\,$ then $\widetilde{A}_{p}\boldsymbol{u}\,=\,A_{p}\boldsymbol{u}\,$ in $\Omega.$
 
 \medskip
 
 Next we study the complex powers of the operator $(I+\widetilde{A}_{p})$ on $\boldsymbol{L}^{p}(\Omega)$.
 
 \begin{prop}\label{rmkI+Aptilde}
Let $0<\alpha<1$ be fixed and let $z\in\mathbb{C}$ such that $-1<-\alpha<\mathrm{Re}\,z<0$. There exist $0<\theta<\pi/2$ and $C>0$ such that 
\begin{equation}\label{complexpower3I+Aptilde}
\Vert(I+\widetilde{A}_{p})^{z}\Vert_{\mathcal{L}(\boldsymbol{L}^{p}(\Omega))}\,\leq\,C(\alpha,\Omega,p)\,e^{\vert\mathrm{Im}\,z\vert\,\theta_{0}},
\end{equation}
where the constant $C=C(\alpha,\Omega,p)>0$ is independent of $\mathrm{Re}\,z$. Moreover when $\Omega$ is strictly star shaped with respect to one of its points, then the following estimate holds
\begin{equation}\label{estimpuI+Aptilde33}
\forall\,\mu>0,\quad\Vert(I+\mu^{2}\,\widetilde{A}_{p})^{z}\Vert_{\mathcal{L}(\boldsymbol{L}^{p}(\Omega))}\,\leq\,C(\alpha,\Omega,p)\,e^{\vert\mathrm{Im}\,z\vert\,\theta_{0}},
\end{equation}
with a constant $C(\alpha,\Omega,p)$ independent of $\mu$ and $\mathrm{Re}\,z$.
\end{prop}

\begin{proof}
Consider the resolvent of the operator $\widetilde{A}_{p}$ on $\boldsymbol{L}^{p}(\Omega).$
 We can verify that for a given function $\boldsymbol{f}\in\boldsymbol{L}^{p}(\Omega)$ and for a given $\lambda\in\mathbb{C}^{\ast}$ such that $\mathrm{Re}\,\lambda\geq 0,\,$ there exists a unique function $(\boldsymbol{u},\pi)\in\mathbf{D}(\widetilde{A}_{p})\times W^{1,p}(\Omega)/\mathbb{R}$ solution to the problem : 
\begin{equation*}
 \left\{
\begin{array}{cccc}
\lambda \boldsymbol{u} - \Delta \boldsymbol{u}\,+\,\nabla\pi = \boldsymbol{f},& \mathrm{div}\,\boldsymbol{u} = 0& \qquad \mathrm{in}&\Omega, \\
\boldsymbol{u}\cdot \boldsymbol{n} = 0,&[\mathbb{D}(\boldsymbol{u})\boldsymbol{n}]_{\tau}=\boldsymbol{0} &\qquad\mathrm{on}&\Gamma.
\end{array}
\right.
\end{equation*}
In addition the following estimate holds 
\begin{equation}\label{resolestimAptilde}
\Vert\boldsymbol{u}\Vert_{\boldsymbol{L}^{p}(\Omega)}\,\leq\,\frac{C(\Omega,p)}{\vert\lambda\vert}\Vert\boldsymbol{f}\Vert_{\boldsymbol{L}^{p}(\Omega)}.
\end{equation}
The resolvent set of the operator $\widetilde{A}_{p}$ contains the half-plane $\{\lambda\in\mathbb{C}^{\ast};\,\,\mathrm{Re}\lambda\geq 0\},$ where estimate \eqref{resolestimAptilde} is satisfied. As a result, (see \cite[Chapter VIII, Theorem 1]{Yo}), there exists an angle $0<\theta_{0}<\pi/2$ such that 
the resolvent set of the operator $\widetilde{A}_{p}$ contains the sector $\Sigma_{\theta_{0}}=\big\{\lambda\in\mathbb{C};\,\,\,\vert\arg\lambda\vert\leq\pi-\theta_{0}\big\}$,
where estimate \eqref{resolestimAptilde} is satisfied.

Proceeding in the same way as in the proof of Proposition \ref{I+lap}, we can verify that the operator $I+\widetilde{A}_{p}$ is a non-negative operator with bounded inverse. Thus, for any complex $z\in\mathbb{C}\,$ the complex power $(I+\widetilde{A}_{p})^{z}$ is well defined by the mean of the following Dunford integral :
 \begin{equation*}
(I+\widetilde{A}_{p})^{z}\,=\,\frac{1}{2\,\pi\,i}\,\int_{\Gamma_{\theta_{0}}}(-\lambda)^{z}\,(\lambda\,I+I+\widetilde{A}_{p})^{-1}\,\textrm{d}\,\lambda,
\end{equation*}
with $\Gamma_{\theta_{0}}$ is given by \eqref{gammatheta0}.

We can also verify, (proceeding in the same way as in the proof of Proposition \ref{pureimg1+lap}), that for any  complex $z\in\mathbb{C}$ such that $-1<-\alpha<\mathrm{Re}\,z<0$, with $0<\alpha<1$ fixed one has
\begin{equation*}
\Vert(I+\widetilde{A}_{p})^{z}\Vert_{\mathcal{L}(\boldsymbol{L}^{p}(\Omega))}\,\leq\,C(\alpha,\Omega,p)\,e^{\vert\mathrm{Im}\,z\vert\,\theta_{0}},
\end{equation*}
with a constant $C(\alpha,\Omega,p)$ independent of $\mathrm{Re}\,z$.

Consider now the case where $\Omega$ is strictly star shaped with respect to one of its point. As in the proof of Theorem \ref{Lapimpowerproposition}, using the scaling transformation $S_{\mu}$ given by \eqref{SMU} and the fact that 
\begin{equation*}
\forall\,\mu>0,\quad\mu^{2}\,\widetilde{A}_{p}\,=\,S_{\mu}\,\widetilde{A}_{p}\,S_{\mu}^{-1},\qquad I+\mu^{2}\widetilde{A}_{p}\,=\,S_{\mu}(I+\widetilde{A}_{p})\,S^{-1}_{\mu},
\end{equation*}
one has, for $-1<-\alpha<\mathrm{Re}\,z<0$, with $0<\alpha<1$ is fixed
\begin{equation}\label{estimpuI+Aptilde331}
\Vert(I+\mu^{2}\,\widetilde{A}_{p})^{z}\Vert_{\mathcal{L}(\boldsymbol{L}^{p}(\Omega))}\,\leq\,C(\alpha,\Omega,p)\,e^{\vert\mathrm{Im}\,z\vert\,\theta_{0}},
\end{equation}
with a constant $C(\alpha,\Omega,p)$ independent of $\mu$ and $\mathrm{Re}\,z$.
\end{proof}

\begin{rmk}\label{rmkdomainenondense}
\rm{
Notice that since $\mathbf{D}(\widetilde{A}_{p})$ is not dense in $\boldsymbol{L}^{p}(\Omega),$ estimate \eqref{estimpur1+lap} doesn't necessarily hold for the operator $I+\widetilde{A}_{p}\,$ in $\boldsymbol{L}^{p}(\Omega).$
}
\end{rmk}

\begin{proof}[Proof of Theorem \ref{Lapimpowerproposition} (General case)]
Let $(\Theta_{j})_{j\in J}$ be an open covering of $\Omega$ by a finite number of star-shaped open domains of class $C^{2,1}$ and let us consider a partition of unity $(\varphi_{j})_{j\in J}$ subordinated to the covering $(\Omega_{j})_{j\in J}$, where $\Omega_{j}=\Theta_{j}\cap\Omega\,$ for all $j\in J.\,$ This means that
\begin{equation*}
\forall j\in J,\quad \mathrm{Supp}\varphi_{j}\subset\Omega_{j},\qquad \sum_{j\in J}\varphi_{j}=1,\quad\varphi_{j}\in\mathcal{D}(\Omega_{j}).
\end{equation*}

Now, let $\mu>0\,$ and $\,z\in\mathbb{C}\,$ such that $-1<-\alpha<\mathrm{Re}\,z<0$, where $0<\alpha<1$ is fixed. Let also $\boldsymbol{f}\in\boldsymbol{L}^{p}_{\sigma,\tau}(\Omega)$ and write $\boldsymbol{f}$ in the form
\begin{equation*}
\boldsymbol{f}=\sum_{j\in J}\boldsymbol{f}_{j}, \qquad\forall j\in J,\quad\boldsymbol{f}_{j}=\varphi_{j}\boldsymbol{f}.
\end{equation*}
Notice that for all $j\in J,\,$ $\boldsymbol{f}_{j}$ is not necessarily a divergence free function. Observe also that
\begin{equation*}
(I+\mu^{2}A_{p})^{z}\boldsymbol{f}\,=\,(I+\mu^{2}\widetilde{A}_{p})^{z}\boldsymbol{f}\,=\,\sum_{j\in J}(I+\mu^{2}\widetilde{A}_{p})^{z}\boldsymbol{f}_{j}.
\end{equation*}
As a result, using Proposition \ref{rmkI+Aptilde}, one has
\begin{eqnarray}
\Vert(I+\mu^{2}A_{p})^{z}\boldsymbol{f}\Vert_{\boldsymbol{L}^{p}(\Omega)}&\leq&\sum_{j\in J}\Vert(I+\mu^{2}\widetilde{A}_{p})^{z}\boldsymbol{f}_{j}\Vert_{\boldsymbol{L}^{p}(\Omega)}\nonumber\\
&=&\sum_{j\in J}\Vert(I+\mu^{2}\widetilde{A}_{p})^{z}\boldsymbol{f}_{j}\Vert_{\boldsymbol{L}^{p}(\Omega_{j})}.\nonumber
\end{eqnarray}
Since for all $j\in J,\,$ $\Omega_{j}\,$ is strictly star shaped with respect to one of its points, then using estimate \eqref{estimpuI+Aptilde33} one has 
\begin{eqnarray*}
\Vert(I+\mu^{2}A_{p})^{z}\boldsymbol{f}\Vert_{\boldsymbol{L}^{p}(\Omega)}&\leq&\,e^{\vert\mathrm{Im}\,z\vert\,\theta_{0}}\sum_{j\in J}\,C_{j}\,\Vert\boldsymbol{f}_{j}\Vert_{\boldsymbol{L}^{p}(\Omega_{j})}\nonumber\\
&\leq& C(\alpha,\Omega,p)\,e^{\vert\mathrm{Im}\,z\vert\,\theta_{0}}\Vert\boldsymbol{f}\Vert_{\boldsymbol{L}^{p}(\Omega)}
\end{eqnarray*}
with a constant $C(\alpha,\Omega,p)$ independent of $\mu,\,\mathrm{Re}\,z$ and $\boldsymbol{f}$. 

As in the proof of Proposition \ref{pureimg1+lap}, using the fact that for all $\boldsymbol{f}\in\mathbf{D}(A_{p})$ the operator $(I+\mu^{2}\,A_{p})^{z}\boldsymbol{f}$ is analytic in $z,\,$ with $z$ such that $\,-1<\mathrm{Re}z<1,\,$ and using the density of $\mathbf{D}(A_{p})$ in $\boldsymbol{L}^{p}_{\sigma,\tau}(\Omega),\,$ one has for all $s\in\mathbb{R},\,$
\begin{equation}\label{estI+mu2Apisnonetoile}
\Vert(I+\mu^{2}\,A_{p})^{i s}\Vert_{\mathcal{L}(\boldsymbol{L}^{p}_{\sigma,\tau}(\Omega))}\,\leq\,C(\alpha,\Omega,p)\,e^{\vert s\vert\,\theta_{0}},
\end{equation}
with a constant $C(\alpha,\Omega,p)$ independent of $\mu.$ Finally, using estimate (\ref{estI+mu2Apisnonetoile}) and the fact that $$\Big(\frac{1}{\mu^{2}}\,I+A_{p}\Big)^{i\,s}\,=\,\frac{1}{\mu^{2\,i\,s}}\,(I+\mu^{2} A_{p})^{i\,s},$$ we have
\begin{equation}\label{estnu*nonetoile}
\Big\Vert\Big(\frac{1}{\mu^{2}}\,I+A_{p}\Big)^{i\,s}\Big\Vert_{\mathcal{L}(\boldsymbol{L}^{p}_{\sigma,\tau}(\Omega))}\,\leq\,C(\alpha,\Omega,p)\,e^{\vert s\vert\,\theta_{0}}.
\end{equation}
thus the result is proved.
\end{proof}

\textbf{Acknowledgment.} The author wish to thank the Professors C. Amrouche and M. Escobedo for their remarks that help to improve the manuscript.

\end{document}